\documentclass[reqno]{amsart}
\usepackage[foot]{amsaddr}

\usepackage{amsmath,amssymb,amsfonts,mathrsfs}
\usepackage{comment}
\usepackage{bbold}
\usepackage{tikz-cd}
\usepackage{csquotes}
\usepackage[english]{babel}
\usepackage[utf8]{inputenc}
\usepackage[T1]{fontenc}
\usepackage{mathtools}
\usepackage{lmodern}

\usepackage{cancel}

\usepackage[
top=3.2cm,
bottom=3.2cm,
left=3.2cm,
right=3.2cm,
marginparwidth=2.5cm
]{geometry}

\usepackage[colorinlistoftodos]{todonotes}
\usepackage[hypertexnames=false]{hyperref}
\hypersetup{
    colorlinks,
    linkcolor={red!50!black},
    citecolor={blue!50!black},
    urlcolor={blue!80!black},
}

\usepackage[shortlabels]{enumitem}
\setlist[1]{itemsep=2pt}
\usepackage[capitalise]{cleveref}

\usepackage{autonum} %to number only \eqref'ed equations in a cref compatible way

\newtheorem{theorem}{Theorem}[section]
\newtheorem{prop}[theorem]{Proposition}
\newtheorem{corollary}[theorem]{Corollary}
\newtheorem{lemma}[theorem]{Lemma}
\newtheorem{defi}[theorem]{Definition}

\theoremstyle{remark} 
\newtheorem{remark}[theorem]{Remark}

\newtheorem{rmk}[theorem]{Remark}

\newcommand{\ve}{\varepsilon}

\newcommand{\red}{\color{red}}

\newcommand{\cut}{\mathrm{cut}}

\title[]{Classification of K-contact forms and spectral invariants of their sub-Laplacians}

\author[E. Bellini]{Eugenio Bellini}
\address{Dipartimento di Matematica "Tullio Levi-Civita", Universit\`a degli Studi di Padova, via Trieste 63, 35131 Padova (PD), Italy\\ \textbf{\emph{eugenio.bellini@unipd.it}}}

\date{}

\begin{document}
\maketitle
\begin{abstract}
A contact form is called K-contact if its Reeb vector field is Killing with respect to some Riemannian metric. In this paper we classify K-contact forms whose Reeb vector field admits at least one non-periodic orbit, on three-dimensional manifolds. We prove that if a compact three-manifold carries such a contact form, then it is diffeomorphic to a lens space and admits exactly two periodic Reeb orbits, whose periods have irrational ratio. We further classify, up to (global) diffeomorphism, these contact forms in terms of the periods of their closed Reeb orbits. We conclude by relating these periods to spectral invariants of the sub-Laplacian, confirming a conjecture of \cite{ClassicalQuantum23} in the irregular K-contact case. 
\end{abstract}
\tableofcontents
\section{Introduction}\noindent
A smooth one-form $\alpha$ on a three dimensional manifold $M$ is called a contact form if $\alpha\wedge d\alpha\neq 0.$ If $M$ is oriented by $\alpha\wedge d\alpha$, then $\alpha$ is called positive. To any contact form there is an associated vector field $R$ called the Reeb field, uniquely determined by the following two equations
\begin{equation}
	\alpha(R)=1,\qquad d\alpha(R,\cdot\,)=0.
\end{equation}
Contact forms can be divided into two classes: quasi-regular and irregular \cite{BoyerGalicki2008} (the term \enquote{quasi-regular} is used instead of \enquote{regular} because the latter classically refers to Boothby-Wang contact forms \cite{Boothby}). A contact form is called quasi-regular if all the orbits of its Reeb vector field are periodic, and irregular if at least one Reeb orbit is non-periodic. Quasi-regular contact forms are completely classified, and their Reeb vector fields generate locally free circle actions, giving rise to Seifert fibrations \cite{KegelLange2021,CristofaroMazz2020,Sullivan78,Wadsley75}. In particular, any quasi-regular contact form is K-contact, i.e., its Reeb vector field is Killing with respect to a suitable Riemannian metric \cite{Blair,BoyerGalicki2008}.
From this perspective, it appears natural to study irregular contact forms admitting a Reeb-invariant metric, that is, irregular K-contact forms. The purpose of this paper is to classify irregular K-contact forms in dimension three. 

Our first result is the following. 
\begin{theorem}\label{thm:main1}
Let $(M,\alpha)$ be a compact three-dimensional manifold with a positive irregular K-contact form. Then $M$ is diffeomorphic to a lens space. Moreover, the Reeb vector field of $\alpha$ has precisely two periodic orbits, and the ratio of their minimal periods is irrational.
\end{theorem}\noindent
 \noindent
 The proof of \cref{thm:main1} is based on the geometry of tubular neighborhoods of periodic Reeb orbits. 
 Using the Reeb-invariant metric provided by the K-contact condition, we associate to each periodic orbit a volume function that measures the contact volume of its metric tubular neighborhoods (see \cref{ssec:vol_fun}). 
 In the irregular K-contact case we show that this function is smooth and that its critical set consists entirely of periodic Reeb orbits. 
 We then prove that the critical set is precisely the union of two periodic orbits and, by gluing their tubular neighborhoods, we deduce that $M$ is diffeomorphic to a lens space. 
 The irrationality of the ratio of the minimal periods follows from a relation between these periods and the eigenvalues of the linearized Poincar\'e return maps along the periodic orbits.

Next, we would like to state a classification result. Recall that two contact forms $\alpha,\beta$ are called strictly contactomorphic if there exists a diffeomorphism $\varphi:M\to M$ such that 
\begin{equation}
	\varphi^*\alpha=\beta.
\end{equation}
The second result of this paper is the following.
\begin{theorem}\label{thm:classification}
	Let $p,q$ be coprime integers, $p>0$, and let $\tau_0,\tau_1\in\mathbb{R}^+$ be such that $\tau_0/\tau_1\in\mathbb R\setminus\mathbb Q$.
	\begin{enumerate}
		\item[(i)] If $q^2\equiv 1 \,\mathrm{mod}\,p$, then there exists a unique irregular positive $K$-contact form on $L(p,q)$ admitting two periodic Reeb orbits of minimal periods $\tau_0,\tau_1$, up to strict contactomorphism.
		\item[(ii)] If $q^2\not\equiv 1 \,\mathrm{mod}\,p$, then there exist precisely two irregular positive $K$-contact forms on $L(p,q)$ admitting two periodic Reeb orbits of minimal periods $\tau_0,\tau_1$, up to strict contactomorphism.
	\end{enumerate}
\end{theorem}

Several examples of contact forms whose Reeb flow has exactly two periodic orbits appear in the literature; for example, irrational ellipsoids in $\mathbb C^2$ provide one such family \cite[Ex.\,1.1]{2orbits23}. 
More generally, contact forms with exactly two periodic Reeb orbits have been extensively studied: it is known that $2$ is the minimal number of periodic Reeb orbits on a three-dimensional closed contact manifold \cite{CrisHutc2016,BrokenBook23,TorsionForms2019}, and that if a closed manifold carries a contact form with exactly two periodic Reeb orbits, then it is diffeomorphic to a lens space \cite{2orbits23}.
We stress that the contact forms of \cref{thm:classification} do not provide an exhaustive list of contact forms having exactly two periodic Reeb orbits. 
Indeed, there exist contact forms on $S^3$ whose Reeb flow has precisely two periodic orbits while all others are dense (see \cite{Katok24,FayadKatok24,AlbersGeigesKai2022}, \cite[Rmk.\,1.7]{2orbits23} and \cite[Rmk.\,1.6]{CristofaroMazz2020}), a feature which is not shared by the forms of \cref{thm:classification} (cf.~\cref{sec:tubes}).

\subsection{Application: Reeb periods are sub-Riemannian spectral invariants} 
A classical theme in Riemannian geometry is the study of the relationships between the geometry and topology of a manifold and the spectral properties of its Laplace–Beltrami operator \cite{Rosenberg}. \newline\indent
In the presence of a contact structure, the picture is enriched by the interplay between the contact structure and the metric: when a contact manifold $(M,\alpha)$ is equipped with a Riemannian metric $g$ one can consider, alongside the Laplace–Beltrami operator, the \emph{sub-Laplacian}. This is a second-order hypoelliptic operator naturally associated to the sub-Riemannian structure determined by the contact distribution and the metric \cite{Agrachev}.  %(see \cite{Agrachev} for an introduction to sub-Riemannian geometry). 

To recall the definition of the sub-Laplacian, we first define the sub-Riemannian gradient of $f\in C^\infty(M)$, denoted with $\nabla^{sR} f$, as the unique smooth section of $\xi=\ker\alpha$ satisfying 
\begin{equation}
	df(Y)=g(\nabla^{sR}f,Y),\qquad \forall\, Y\in\Gamma(\xi).
\end{equation}
The sub-Laplacian of $f\in C^\infty(M)$ is then defined as the divergence of its horizontal gradient 
\begin{equation}
	\Delta^{sR} f=\mathrm{div}(\nabla ^{sR} f),
\end{equation}
where the divergence is taken with respect to the volume form $\alpha\wedge d\alpha$. As the Laplace–Beltrami operator plays a central role in Riemannian geometry, the study of the sub-Laplacian has become a central topic in the sub-Riemannian setting. See, for instance, \cite{Hormander67,Strichartz86} for matters of regularity, \cite{ClassicalQuantum23,SpectralAsymp2022,SpectralAsym1,BarBoscNeel2012,Barilari2013,BarBoscNeel2012,BenArous88,BenArous89} for heat kernel asymptotics on and outside the diagonal, and \cite{AgrachevRizziRossi24,RizziRossi21,Rossi21,TysonWang18} for heat content.

We are interested in the connections linking the spectrum of the sub-Laplacian and the dynamical properties of the Reeb vector field. One way of recovering spectral information is to study the explicit form of the coefficients appearing in the small-time asymptotic expansion of the trace of the \emph{heat kernel}, the latter being the fundamental solution $P\in C^{\infty}(\mathbb R^{+}\times M\times M)$ of the heat equation:
\begin{equation}
		\begin{cases}
			\partial_tP(t,x,y)-\Delta^{sR}_xP(t,x,y)=0,\\
			\lim_{t\to 0}\int_MP(t,x,y)f(y)\alpha\wedge d\alpha(y)=f(x),\qquad \forall f\in C_c^{\infty}(M),\,x\in M.
		\end{cases}
\end{equation}
A more detailed definition is given in \cref{sec:trace_heat}. By the small time asymptotic of the heat trace we mean the asymptotic expansion, as $t\to 0$, of the function
\begin{equation}
	t\mapsto \int_M P(t,x,x)\alpha\wedge d\alpha(x).
\end{equation}
The study of such expansion is a classical topic in Riemannian geometry \cite{Rosenberg, ShapeDrum}, and, more recently, it has been explored also in the sub-Riemannian setting \cite{ClassicalQuantum23, SpectralAsym1,Barilari2013,BenArous89,BenArous88}.

For quasi-regular contact forms the Reeb flow induces a Seifert fibration (see \cref{app:siefert}). The first two coefficients of the small time asymptotic expansion of the heat trace can be computed in terms of the invariants of such fibration.

In \cref{sec:trace_heat} we prove the following proposition.

\begin{prop}\label{prop:spectral_inv}
Let $(M,\alpha,g)$ be a compact contact Riemannian manifold where $\alpha$ is positive and quasi-regular and $g$ is Reeb invariant. Then the Reeb flow induces a Seifert fibration of $M$, in particular all but finitely many Reeb orbits have the same minimal period $\tau>0$ and 
\begin{equation}\label{eq:expansion_quasi_reg}
	\int_M  P(t,x,x)\alpha\wedge d\alpha(x)=\frac{\tau}{16t^2}\left(-e(M)+2\pi\chi_{\mathrm{orb}}(\mathcal O)t+O(t^2)\right),
\end{equation}
where $e(M)$ the Euler number of the Seifert fibration and $\chi_{\mathrm{orb}}(\mathcal O)$ is the orbifold Euler characteristic of the base orbifold.
\end{prop}

In the irregular K-contact case, the Reeb flow does not induce a Seifert fibration of $M$. Nonetheless, the first two coefficients of the heat trace expansion can be computed explicitly. By approximating irregular contact forms with quasi-regular ones and showing that the first two coefficients in the expansion \eqref{eq:expansion_quasi_reg} converge, we obtain the following result.
\begin{theorem}\label{thm:spectral_inv}
	Let $(M,\alpha,g)$ be a compact contact Riemannian manifold where $\alpha$ is a positive irregular K-contact form and $g$ is a Reeb invariant metric. According to \cref{thm:main1} there exist $p,q\in\mathbb Z$ co-prime integers such that $M$ is diffeomorphic to $L(p,q)$ and contains exactly two periodic Reeb orbits. Let $\tau_0,\tau_1$ be their minimal periods, then 
	\begin{equation}
		\int_M  P(t,x,x)\alpha\wedge d\alpha(x)=\frac{1}{16t^2}(p\tau_0\tau_1+2\pi(\tau_0+\tau_1)t+O(t^2)).
	\end{equation}
\end{theorem}
It has been conjectured in \cite{ClassicalQuantum23} that the periods of closed orbits of the Reeb vector field are spectral invariants of the sub-Laplacian. \cref{thm:spectral_inv} explicitly confirms such conjecture, for irregular K-contact forms. \cref{thm:spectral_inv} and \cref{prop:spectral_inv} show that contact invariants can be recovered from the spectrum of the sub-Riemannian Laplacian. Thus, they may be contextualized in a growing line of research where contact topology is investigated through the use of sub-Riemannian methods, see for instance \cite{ABBR2024,BellBosc2023,Bell2023,BarBoscCann2022,BarBesLer2020,BarBelPina2025}.
\subsection{Structure of the paper} The paper is organized as follows. In \cref{sec:preliminaries} we review some preliminaries on contact Riemannian manifolds and study their symmetry groups in the K-contact case. In \cref{sec:lens} we recall the definition of lens spaces and we describe them as the union of two charts. In \cref{sec:tubes} we study the normal exponential map from a periodic Reeb orbit in K-contact manifolds. In \cref{sec:char_K_cont} we apply the results of \cref{sec:tubes} and prove \cref{thm:main1}. In \cref{sec:K-cont-Lens} we prove \cref{thm:classification}. Finally, \cref{sec:trace_heat} is dedicated to the proof of \cref{prop:spectral_inv} and \cref{thm:spectral_inv}.
\section{K-contact Riemannian manifolds}\label{sec:preliminaries}
In this section we review some preliminaries on K-contact forms and we study their symmetry groups. In particular we establish relations between periodic Reeb orbits and the critical set of contact momentum maps. Throughout the section $M$ denotes a closed oriented three dimensional smooth manifold.
\subsection{Contact forms and compatible metrics}
A smooth plane field $\xi$ on $M$ is called a contact structure if there exists a differential form $\alpha$, called a contact form, such that 
\begin{equation}
    \xi=\ker\alpha,\qquad \alpha\wedge d\alpha \neq 0.
\end{equation}
Given a never vanishing smooth function $f\in C^\infty(M)$, the rescaled form $f\alpha$ is also a contact form. However, the sign of $\alpha\wedge d\alpha$ does not depend on the scaling factor. A contact form is called positive if $\alpha\wedge d\alpha>0$. To any contact form there is an associated vector field $R$ called the Reeb field, uniquely determined by the following two equations
\begin{equation}\label{eq:def_reeb}
    \alpha(R)=1,\qquad d\alpha(R,\cdot\,)=0.
\end{equation}
\begin{defi}\label{def:compatible}
  A Riemannian metric $g$ is called compatible with a positive contact form $\alpha$ if
  \begin{equation}\label{eq:comp_metric}
      g(R,\cdot\,)=\alpha(\, \cdot\,),\qquad \mathrm{vol}_g=\alpha\wedge d\alpha,
  \end{equation}
  where $\mathrm{vol}_g$ is the volume form of the oriented Riemannian manifold $(M,g)$. The triple $(M,\alpha,g)$ is called a contact Riemannian manifold.
\end{defi}
For a fixed contact form $\alpha$, there exist several metrics satisfying \eqref{eq:comp_metric}. 
\begin{prop}\label{prop:comp_metrics}
    Let $\eta$ be a bundle metric on $\xi$ satisfying
    \begin{equation}\label{eq:sr_comp_metric}
        d\alpha|_\xi=\mathrm{area}_\eta,
    \end{equation}
    where $\mathrm{area}_\eta$ is the area form of the oriented Euclidean bundle $(\xi,\eta)$, the orientation being the one determined by $d\alpha|_\xi$. Then the Riemannian metric $g$ defined by 
    \begin{equation}
    	g=p^*\eta+\alpha\otimes\alpha,
    \end{equation}
	where $p:TM\to \xi$ denotes the linear projection $p(v)=v-\alpha(v)R$, is compatible with $\alpha$.
\end{prop}
\begin{proof}
    We check the two conditions of \eqref{eq:comp_metric}. The first one follows from $\alpha(R)=1$ and $p(R)=0$:
    \begin{equation}
        g(R,\cdot)=\eta(p(R),p(\cdot))+\alpha(R)\alpha(\cdot)=\alpha(\cdot).
    \end{equation}
    The latter equality implies that $R$ is a unit normal to $\xi$, and that $p:TM\to \xi$ is the orthogonal projection to $\xi$ determined by the metric $g$. Therefore
    \begin{equation}
    	\mathrm{vol}(g)=\alpha\wedge p^*(\mathrm{area}_{g|_{\xi}}).
    \end{equation}
	Moreover, $g|_{\xi}=\eta$ and $\mathrm{area}_\eta=d\alpha|_{\xi}$, therefore 
	\begin{equation}
		\mathrm{vol}(g)=\alpha\wedge p^*(\mathrm{area}_{g|_{\xi}})=\alpha\wedge p^*(\mathrm{area}_\eta)=\alpha\wedge p^*(d\alpha|_{\xi})=\alpha\wedge d\alpha,
	\end{equation}
	where we have used the fact that $p^*(d\alpha|_{\xi})=d\alpha$ since, by definition of Reeb field, $d\alpha(R,\cdot)=0$.
\end{proof}
\begin{remark}
    Several notions of compatibility between Riemannian metrics and contact structures appear in the literature. The one that we give here corresponds to the \enquote{strong compatibility} of \cite{Massot}, see also \cite{Blair}. Equation  \eqref{eq:sr_comp_metric} expresses the sub-Riemannian compatibility condition, see \cite{Agrachev,ABBR2024}.
\end{remark}
We now introduce $K$-contact forms.
\begin{defi}
    A positive contact form $\alpha$ on $M$ is called $K$-contact if the associated Reeb vector field is a Killing vector field for a  smooth Riemannian metric $g$. In other terms
    \begin{equation}
    	\mathcal L_Rg=0.
    \end{equation}
	The couple $(M,\alpha)$ is called a K-contact manifold.
\end{defi}
The next proposition shows that the Riemannian metric $g$ in the previous definition can be chosen to be compatible with $\alpha$.
\begin{prop}\label{prop:comp_k_cont}
	Let $\alpha$ be a positive $K$-contact form with Reeb vector field $R$. There exists a Riemannian metric $g$, compatible with $\alpha$ in the sense of \cref{def:compatible}, such that $\mathcal L_Rg=0$.
\end{prop}
\begin{proof}
Let $\bar g$ be a Riemannian metric such that $\mathcal L_R\bar g=0$. We endow $\xi$ with the orientation induced by $d\alpha |_{\xi}$. Let $\sigma$ be the positive area form that $\bar{g}|_{\xi}$ induces on the oriented bundle $\xi$. Observe that $\mathcal L_R\sigma=0$. There exists a smooth positive function $f:M\to \mathbb R$ such that 
\begin{equation}
	d\alpha|_{\xi}=f\sigma.
\end{equation}
Therefore the bundle metric $\eta:=f\bar g|_{\xi}$ satisfies $\mathrm{area}_\eta=d\alpha|_{\xi}$ and $\mathcal L_R\eta=0$.
Let $p:TM\to \xi$ be the linear projection $p(v)=v-\alpha(v)R$. We define the Riemannian metric $g$ as
\begin{equation}
	g=p^*\eta+\alpha\otimes\alpha.
\end{equation}
Since $\mathcal L_R\alpha=\mathcal L_R\eta=0$, one has $\mathcal L_Rg=0$. The compatibility with $\alpha$ follows from \cref{prop:comp_metrics}.
\end{proof}
\begin{defi}
	The triple $(M,\alpha,g)$ is called a K-contact Riemannian manifold if $M$ is a three manifold, $\alpha$ is a positive K-contact form, and $g$ is a compatible metric satisfying $\mathcal L_Rg=0$.
\end{defi}
\subsection{Rotation number}
In this section we introduce an invariant for periodic orbits of the Reeb vector field associated with a K-contact form. 

Let $(M,\alpha)$ be a K-contact manifold and let $\ell\subset M$ be a periodic orbit of the Reeb vector field of minimal period $\tau>0$. That is to say that, for any $q\in \ell$, $\tau$ is the smallest positive real number such that 
\begin{equation}
	e^{\tau R}(q)=q.
\end{equation}
Let $g$ be a Riemannian metric invariant under the flow of the Reeb vector field. Let $\xi=\ker\alpha$ and note that for any fixed $q\in \ell$, the linear map
\begin{equation}\label{eq:first_return}
	e^{\tau R}_*:\xi_q\to \xi_q,
\end{equation}
is an isometry of 2-dimensional euclidean vector spaces, the metric being $g|_\xi$. Hence, its eigenvalues are of the form $e^{\pm i2\pi\phi}$, with $\phi\in\mathbb R$.
\begin{defi}\label{def:rotation_num}
	The rotation number $\phi\in\mathbb R/\mathbb N$ associated with the periodic Reeb orbit $\ell$ is defined by the following equation 
	\begin{equation}
		\lambda_{\pm}(\ell)=e^{\pm i2\pi\phi},
	\end{equation}
	where $\lambda_{\pm}(\ell)$ are the eigenvalues of \eqref{eq:first_return}.
\end{defi}
\begin{remark}\label{rmk:explicit_rotation}
	Let $\phi$ be the rotation number of $\ell$ and let $v_1,v_2$ be an orthonormal frame for $g|_{\xi_q}$. Since the Reeb flow preserves the Reeb vector field, the linear map $e^{\tau R}_*:T_qM\to T_qM$, written with respect to the ordered basis $v_1,v_2,R$, has the following associated matrix
	\begin{equation}
		e^{\tau R}_*=\begin{pmatrix}
			\cos(2\pi\phi) & -\sin(2\pi\phi) & 0\\
			\sin(2\pi\phi) & \cos(2\pi\phi) & 0\\
			0 & 0 & 1
		\end{pmatrix}.
	\end{equation}
\end{remark}
Although being a semi-local invariant, the rotation number of a single Reeb orbit determines the behavior of all the other ones.
\begin{lemma}\label{lem:rot_determines_flow}
	Let $(M,\alpha)$ be a K-contact manifold and let $\ell\subset M$ be a periodic orbit of the Reeb vector field having rotation number $\phi$. If $\phi\in\mathbb Q$, then all Reeb orbits are periodic. Moreover, if $\ell$ has minimal period $\tau$ and $\phi=n/m$ with $m,n\in\mathbb N$ coprime, then there exists a neighborhood $O$ of $\ell$ such that for any $x\in O$ the Reeb orbit containing $x$ has minimal period $m\tau$. 
\end{lemma}
\begin{proof}
	Let $g$ be a Riemannian metric invariant under the flow of the Reeb vector field $R$. Let $x\in M$ and let $\sigma:[0,1]\to M$ be a length minimizing geodesic joining $x$ to $\ell$, i.e.
	\begin{equation}
		\sigma(0)\in\ell,\qquad\sigma(1)=x,\qquad \mathrm{length}(\sigma)=\inf_{y\in\ell}d(x,y), 
	\end{equation}
	where $d(\cdot,\cdot)$ is the Riemannian distance. Now consider the 1-parameter family of curves $$\sigma_t=e^{tR}\circ\sigma,\qquad t\in\mathbb R.$$
	Since the Reeb flow acts by isometries, each $\sigma_t$ is a length minimizing geodesic. By hypothesis there exists $m,n\in\mathbb N$ coprime such that $\phi=n/m$, therefore by definition of rotation number one has $$e^{m\tau R}_*=\left(e^{\tau R}_*\right)^m=Id,$$ where $\tau$ is the minimal period of $\ell$ (cf. \cref{rmk:explicit_rotation}). It follows that $\dot\sigma(0)=\dot\sigma_{m\tau}(0)$, which implies $\sigma\equiv\sigma_{m\tau}$, since both curves are geodesics. Consequently
	\begin{equation}
		e^{m\tau R}(x)=e^{m\tau R}\circ\sigma(1)=\sigma_{m\tau}(1)=\sigma(1)=x.
	\end{equation}
	Therefore the Reeb orbit through $x$ is periodic of minimal period not greater than $m\tau$.  If in the above reasoning we choose $x$ sufficiently close to $\ell$, then there is a unique length minimizing geodesic joining $x$ to $\ell$ (cf. \cref{ssec:tube_Riemann}), which is $\sigma$. Therefore, for such an $x$ we have
	\begin{equation}\label{eq:compute_period}
		e^{t R}(x)=x\iff e^{tR}\circ\sigma(1)=\sigma(1)\iff e^{tR}\circ\sigma\equiv \sigma\iff e^{tR}_*\dot\sigma(0)=\dot\sigma(0),
	\end{equation}
	where uniqueness of the minimizing geodesic has been used in the second equivalence. Since $\ell$ has minimal period $\tau$ and rotation number $\phi=n/m$, the minimal $t$ for which \eqref{eq:compute_period} holds is $m\tau$.
\end{proof}
\begin{lemma}\label{lem:R(f)=0}
	Let $(M,\alpha)$ be a K-contact manifold and let $f\in C^{\infty}(M)$ be a Reeb invariant function, i.e. $df(R)=0$. Assume that there exists a periodic Reeb orbit $\ell$ contained in a regular level set of $f$. Then $\ell$ has zero rotation number. In particular all Reeb orbits are periodic.
\end{lemma}
\begin{proof}
	Let $g$ be a Riemannian metric invariant under the flow of $R$ and compatible with $\alpha$. Let $X=\mathrm{grad}\,f$ be the gradient of $f$. Note that $X\in\Gamma(\xi)$, indeed, making use of \eqref{eq:comp_metric} we find
	\begin{equation}
		\alpha(X)=g(X,R)=df(R)=0.
	\end{equation}
	Let $c\in\mathbb R$ and $S=f^{-1}(c)$ be the regular level set of $f$ containing $\ell$. Thus, $X$ does not have any zeroes along $S$. Hence, we can define the following orthonormal frame for $\xi|_S$
	\begin{equation}\label{eq:e_1,e_2}
		e_1=|X|^{-1}X,\qquad e_2=e_1^{\perp_\xi},
	\end{equation}
	where $e_1^{\perp_\xi}$ denotes the oriented orthonormal complement to $e_1$ in $\xi$. Since $R$ is a Killing field and $df(R)=0$ it holds $[R,X]=0$, and consequently
	\begin{equation}\label{eq:[R,e_i]=0}
		[R,e_1]=[R,e_2]=0.
	\end{equation}
	Since $\ell\subset S$, the orthonormal frame $e_1,e_2$ is well defined along $\ell$. Let $\tau>0$ be the minimal period of $\ell$, then by \eqref{eq:[R,e_i]=0} we have 
	\begin{equation}
		e^{\tau R}_*e_1=e_1,\qquad e^{\tau R}_* e_2=e_2,
	\end{equation}
	which shows that $e^{\tau R}_*=Id$. Hence the rotation number is zero and, by \cref{lem:rot_determines_flow}, all Reeb orbits are periodic.
\end{proof}
\subsection{Contact isometry group and momentum maps}
We introduce the K-contact isometry group and the K-contact momentum maps.
\begin{defi}
	A contact isometry of a contact Riemannian manifold $(M,\alpha,g)$ is a diffeomorphism $\varphi:M\to M$ such that $\varphi^*\alpha=\alpha,\varphi^*g=g$.
	The contact isometry group $G$ is defined as the collection of the contact isometries
	\begin{equation}
		G=\{\varphi\in\mathrm{Diff}(M)\,:\, \varphi^*g=g,\,\,\,\varphi^*\alpha=\alpha \}.
	\end{equation}
\end{defi}
\begin{remark}\label{rmk:contact_group}
	The contact isometry group $G$ of a closed contact Riemannian manifold $(M,\alpha,g)$ is a compact Lie group. Indeed, $G$ is a closed subgroup of the isometry group of the compact Riemannian manifold $(M,g)$, which is a compact Lie group by Myers–Steenrod theorem. Furthermore, if $(M,\alpha,g)$ is K-contact then the flow of the Reeb vector field defined by 
	\begin{equation}
		\{e^{tR}\in \mathrm{Diff}(M)\,:\, t\in\mathbb R\},
	\end{equation}
	is a central subgroup of $G$. Indeed, by definition of $G$, $\varphi^*\alpha=\alpha$ for all $\varphi\in G$. Since the Reeb vector field is uniquely determined by $\alpha$, one has $\varphi_*R=R$ for all $\varphi\in G$.
\end{remark}
\noindent 
In the remaining part of the section, given a group $G$ acting on a manifold $M$ and a subset $A\subset M$ we denote with $G_A$ the stabilizer of $A$ in $G$, i.e.
\begin{equation}
	G_A=\{\varphi\in G\,:\, \varphi(A)\subset A\}.
\end{equation}
\begin{lemma}\label{lem:G_is_T}
	Let $(M,\alpha,g)$ be a closed K-contact Riemannian manifold with contact isometry group $G$. Let $S\subset M$ be an oriented compact surface invariant under the flow of the Reeb field. Then $(S, g|_{S})$ is a flat Riemannian torus and, if there are no periodic Reeb orbits on $S$, the map 
	\begin{equation}\label{eq:G-hom}
		G_S\to\mathrm{Isom}_0(S),\qquad \varphi\mapsto \varphi|_{S},
	\end{equation}
	is a group isomorphism, where $\mathrm{Isom}_0(S)$ is the identity component of the isometry group of $(S,g|_S)$.
\end{lemma} 
\begin{proof}
	Let $R^{\perp_S}\in\mathrm{Vec}(S)$ be the oriented orthonormal complement of $R|_{S}$ taken w.r.t. $g|_{S}$. Since the Reeb flow acts by contact isometries and fixes $S$, we have $[R,R^{\perp_S}]=0$. Thus, the compact Riemannian surface $(S,g|_S)$ carries a global commuting orthonormal frame, therefore $S$ is a torus and $g|_{S}$ is a flat Riemannian metric.
	We now prove the injectivity of \eqref{eq:G-hom}. If $\varphi\in G_{S}$ satisfies $\varphi|_S=Id_S$, then for each $q\in S$ the map
	\begin{equation}
		d_q\varphi|_S:T_qS\to T_qS,
	\end{equation}
	is the identity map. Let $v_1$ denote the oriented unit normal to $T_qS$. Since $\varphi$ is an orientation preserving isometry fixing $S$ one has $d_q\varphi v_1=v_1$. It follows that $d_q\varphi:T_qM\to T_qM$ is the identity map. Summarizing $\varphi$ is an isometry of a Riemannian manifold satisfying 
	\begin{equation}
		\varphi(q)=q,\qquad d_q\varphi=Id_{T_qM},
	\end{equation}
	therefore $\varphi=Id_M$.
	Assume that $R$ has no periodic orbit in $S$. Then, since $S$ is a torus, the generalized {P}oincar\'{e}-{B}endixson theorem \cite{Schwartz63} implies that every orbit of $R|_{S}$ is dense in $S$. For each $q\in S$ we have
	\begin{equation}
		\{e^{tR}(q)\,:\, t\in\mathbb R\}\subset G_S\cdot q:=\{\varphi(q)\,:\, \varphi\in G_S\}\subset S.
	\end{equation}
	Thus, $G_S\cdot q$ is dense in $S$.
	Since $G_S$ is compact, being the stabilizer of a compact set in a compact group, then $G_S\cdot q$ is closed. Therefore $G_S\cdot q=S$. Thus $G_S$ is a closed subgroup of $\mathrm{Isom}_0(S)$, which is a torus, acting transitively on $S$. The only closed subgroups of a torus are subtori. Since proper subtori cannot act transitively, we deduce that $G_S=\mathrm{Isom}_0(S)$ and \eqref{eq:G-hom} is surjective.
\end{proof}
We now introduce K-contact momentum maps and show that their critical sets are foliated by Reeb orbits.
\begin{defi}
	Let $(M,\alpha,g)$ be a K-contatc Riemannian manifold with contact isometry group $G$. We say that a vector field $Y\in\mathrm{Vec}(M)$ is K-contact if $$e^{t Y}\in G,\qquad \forall\,t\in\mathbb R.$$ A smooth map $f:M\to \mathbb R$ is called a K-contact momentum map if there exists a K-contact vector field $Y$ such that 
	\begin{equation}
		df+\iota_Yd\alpha=0.
	\end{equation}
\end{defi}
\begin{remark}\label{rmk:contact_Kill}
	K-contact vector fields form a real Lie algebra of the same dimension as $G$.
\end{remark}
\begin{remark}\label{rmk:Killing_extension}
	Under the hypothesis of \cref{lem:G_is_T}, any Killing vector field for $S$ extends to a global contact Killing vector field for $M$. Indeed let $X\in \mathrm{Vec}(S)$ be a Killing vector field for $S$, then by \cref{lem:G_is_T}, there exists a 1-parameter group of contact isometries of $M$, $\{\varphi_t\}$, such that 
	\begin{equation}
		\varphi_t|_S=e^{tX},\qquad t\in\mathbb R.
	\end{equation} 
	The extension of $X$ is the vector field generating the 1-dimensional group $\{\varphi_t\}\subset \mathrm{Diff}(M)$.
\end{remark}
\begin{remark}\label{rmk:deformation}
	Let $(M,\alpha,g)$ be a compact K-contact Riemannian manifold and let $\mu:M\to \mathbb R$ be a K-contact momentum map. Then, for $\ve\in\mathbb R$ sufficiently small the form 
	\begin{equation}
		\alpha_{\ve}=\frac{\alpha}{1+\ve \mu},
	\end{equation}
	is also a K-contact form. Indeed, let $Y$ be the K-contact vector field such that $d\mu+\iota_Yd\alpha=0$, then
	\begin{equation}
		d(\mu-\alpha(Y))=d\mu-\mathcal L_Y\alpha+\iota_Yd\alpha=d\mu+\iota_Yd\alpha=0.
	\end{equation}
	 Thus, up to an additive constant we may assume $\mu=\alpha(Y)$. Then, the Reeb field of $\alpha_\ve$ is 
	\begin{equation}
		R_\ve=R+\ve Y.
	\end{equation}
	To see this, first note that, by $d\mu+\iota_Yd\alpha=0$ and $\mu=\alpha(Y)$ one obtains
	\begin{equation}
	 d\mu(R_\ve)=-d\alpha(Y,R+\ve Y)=0,\qquad \alpha(R_\ve)=1+\ve\mu,\qquad \iota_{R_\ve}d\alpha=-\ve d\mu.
	\end{equation}
	It follows that $\alpha_\ve(R_\ve)=1$ and that 
	\begin{equation}
		\begin{aligned}
			\iota_{R_\ve}d\alpha_\ve&=-\frac{\ve\iota_{R_\ve}d\mu\wedge \alpha}{(1+\ve\mu)^2}+\frac{\iota_{R_\ve}d\alpha}{1+\ve\mu}=\frac{\ve(1+\ve\mu)d\mu}{1+\ve\mu}-\frac{\ve d\mu}{1+\ve\mu}=0.
		\end{aligned}
	\end{equation}
	Since both $Y$ and $R$ are K-contact vector fields, it follows that $\mathcal L_{R_\ve}g=0$ and $\alpha_\ve$ is K-contact. Furthermore, starting from $g$ and arguing as in \cref{prop:comp_k_cont}, one can build a $R_\ve$-invariant metric $g_\ve$ which is compatible with $\alpha_\ve$ and satisfying 
	\begin{equation}
		\lim_{\ve\to 0}g_\ve=g,
	\end{equation}
	where the limit is intended in $C^\infty$ topology. Note that $\mu$ is also a K-contact momentum map for the perturbed structure $(M,\alpha_\ve,g_\ve)$.
\end{remark}
\begin{theorem}\label{thm:momentum_maps}
	Let $(M,\alpha,g)$ be a compact K-contact Riemannian manifold and let $f:M\to \mathbb R$ be a non-constant K-contact momentum map. Then the critical set of $f$ is a union of periodic Reeb orbits.
\end{theorem}
\begin{proof}
Observe that $f$ is invariant under the Reeb flow and under the flow of $Y$, indeed 
\begin{equation}\label{eq:Reeb_in_f}
	df(R)=-d\alpha(Y,R)=0,\qquad df(Y)=-d\alpha(Y,Y)=0.
\end{equation}
Let $c>0$ be a regular value of $f$, which exists since $f$ is non constant, and let $S=f^{-1}(c)$. If $S$ contains a periodic Reeb orbit then by \cref{lem:R(f)=0}, all Reeb orbits are periodic, since by \eqref{eq:Reeb_in_f}, $f$ is invariant under the flow of $R$. In this case the proof is concluded. 

Otherwise, assume that $S$ contains no periodic Reeb orbits. Then, applying \cref{lem:G_is_T} we deduce that $H:=G_{S}$ is isomorphic to a torus. Let now $q\in M$ be a critical point of $f$. It is sufficient to show that the Reeb orbit containing $q$ is periodic. Let $Y$ be the K-contact vector filed such that 
\begin{equation}
	df+\iota_Yd\alpha=0.
\end{equation}
Since $\ker d\alpha=\mathrm{span}\{R\}$ and $d_qf=0$, there exists a constant $b$ such that 
\begin{equation}
	Y|_q=bR|_q.
\end{equation}
Note that the vector field
\begin{equation}
	Z:=Y-bR,
\end{equation}
is K-contact, tangent to $S$  and vanishes at $q$. Let $$H_q=\{\varphi\in H\,:\, \varphi(q)=q\},$$ which is a closed subgroup of $H$. Note that 
\begin{equation}
	1\leq 	\dim H_q\leq 2.
\end{equation}
The lower bound follows from the inclusion $\{e^{tZ}\}_{t\in\mathbb R}\subset H_q$, $t\in\mathbb R$. The upper bound follows from $H_q\subset H$ and the fact that $\dim H=2$, since $H$ is a torus. Note that $H_q$ cannot be $2$-dimensional, for otherwise it would hold $H_q=H$, while for $t\in\mathbb R$ sufficiently small one has $e^{tR}\subset H\setminus H_q$, since the Reeb vector field never vanishes and its small times flow fixes no point. Thus $\dim H_q=1$ and, being $H_q$ a closed subgroup of the torus $H$, $H/H_q\simeq S^1$, where $H/H_q$ denotes the quotient of $H$ by $H_q$ . Therefore 
\begin{equation}
	H\cdot q=\{\varphi(q)\,:\, \varphi\in\ H\}\simeq H/H_q\simeq S^1,
\end{equation}
where the diffeomorphism between $H/H_q$ and $H\cdot q$ is given by 
\begin{eqnarray}
	H/H_q\to H\cdot q,\qquad [\varphi]\mapsto \varphi(q).
\end{eqnarray}
Since the Reeb field is tangent to $H\cdot q$, and the latter is 1-dimensional, then $H\cdot q$ is actually a periodic orbit of the Reeb field. 
\end{proof}
The following fact holds in much greater generality \cite{Taubes2007}. Here we prove it in the much simpler specific setting of K-contact geometry.
\begin{corollary}\label{thm:one_periodic_orbit}
	Let $\alpha$ be a K-contact form on a compact 3-manifold $M$. Then the associated Reeb vector field $R$ admits at least one periodic orbit.
\end{corollary}
\begin{proof}
Let $g$ be a K-contact metric compatible with $\alpha$ and let $G$ be the contact isometry group of $(M,\alpha,g)$. Then, according to \cref{rmk:contact_group}, $G$ is a compact Lie group and the Reeb flow $\Phi(R)=\{e^{tR}\mid t\in\mathbb R\}$, is a central subgroup of $G$. If $\dim(G)=1$, then $G=\Phi(R)\simeq S^1$, where the latter isomorphism equivalence follows from the classification of compact 1-dimensional Lie groups. In such case all Reeb orbits are periodic and the proof is concluded.

Assume that $\dim(G)>1$. \begin{comment}Let $\mathfrak g$ be the Lie algebra of $G$. For any $v\in \mathfrak g$ we define the vector field $\bar v\in \mathfrak X(M)$ as
\begin{equation}
	\bar v_q:=\frac{d}{dt}\left(\exp(tv)\cdot q\right)|_{t=0},\qquad \forall\, q\in M,
\end{equation}
where $\exp:\mathfrak g\to G$ is the Lie group exponential map. Let us denote 
\begin{equation}
	K=\{\bar v\in\mathfrak X(M)\,:\, v\in\mathfrak g\}.
\end{equation}
\end{comment}
Let $K\subset \mathrm{Vec}(M)$ be the Lie algebra of K-contact vector fields. According to \cref{rmk:contact_Kill} $\dim K=\dim G>1$. Therefore, there exists $Y\in K$ which is not a multiple of the Reeb field $R$. Since $R$ is transverse to $\xi$, there exist $f\in C^\infty(M)$ and $X\in\Gamma(\xi)$ such that $Y=fR+X$. Condition $\mathcal L_Y\alpha=0$ then reduces to 
\begin{equation}\label{eq:X_def}
	\iota_Xd\alpha+df=0.
\end{equation}
The latter equality implies that $df(R)\equiv 0$ and that $X$ vanishes exactly at critical points of $f$. It follows that, if $f$ were constant then we would have $Y=fR$, contradicting the fact that $Y$ and $R$ are linearly independent. Therefore $f$ is a non-constant momentum map. By compactness of $M$, $f$ has at least one critical point. We conclude applying \cref{thm:momentum_maps}
\end{proof}
\section{Lens spaces}\label{sec:lens}
Given a couple of co-prime integers $(p,q)$, with $p>0$, the lens space $L(p,q)$ is a smooth oriented manifold which can be described as the gluing of two solid tori.  Let $V_0,V_1$ be two copies of the solid torus $D^2\times S^1$ with its natural orientation.  Let $(\theta,z)$ be $2\pi$-periodic coordinates on the boundary $\partial V_j=\partial D^2\times S^1$, and consider the following orientation reversing diffeomorphism 
\begin{equation}\label{eq:Lp_glue}
	\varphi:\partial V_1\to \partial V_0,\qquad \varphi(\theta,z)=(-q\theta+mz,p\theta+sz),
\end{equation}
where $m,s\in\mathbb Z$ satisfy $mp+sq=1$. The lens space $L(p,q)$ is the oriented manifold obtained gluing $V_0$ and $V_1$ along their boundaries, by the diffeomorphism $\varphi$:
\begin{equation}\label{eq:Lp_Hed}
	L(p,q)=V_0\sqcup_{\varphi}V_1.
\end{equation}
Recall that given two topological spaces $X,Y$, two subsets $A\subset X, B\subset Y$ and a continuous map $\varphi:B\to A$ the topological space $X\sqcup_\varphi Y$ is defined as $X\sqcup Y/\sim$ where, for $x\neq y$, $x\sim y$ if and only if $x\in A$, $y\in B$ and $\varphi(y)=x$.

\subsection{Toroidal atlases} In this section we give a two charts description of $L(p,q)$. We denote
\begin{equation}
	U=D^2\times S^1\setminus \partial D^2\times S^1,\qquad \ell=\{0\}\times S^1\subset U.
\end{equation}
On $U\setminus\ell$ we introduce cylindrical coordinates $(r,\theta,z)$ by means of the diffeomorphism 
\begin{equation}\label{eq:cyl_coords_U}
	(0,1)\times (\mathbb R/2\pi\mathbb N)\times (\mathbb R/2\pi\mathbb N)\to U\setminus\ell,\qquad (r,\theta,z)\mapsto (re^{i\theta},z).
\end{equation}
\begin{defi}\label{def:toroidal}
	Let $M$ be a three manifold. A collection of pairs $\{(U_j,\psi_j)\}$ is called a \emph{toroidal atlas} if $\{U_j\}$ is an open cover of $M$ and for each $j$ the map 
	\begin{equation}
		\psi_j:U\to  U_j,
	\end{equation}
	is a diffeomorphism. For each $j$, the couple $(U_j,\psi_j)$ is called a toroidal chart.
\end{defi}
We would like to to describe $L(p,q)$ as the union of two toroidal charts.
\begin{prop}\label{prop:charts}
	Let $M$ be a smooth three manifold admitting a toroidal atlas composed by two charts $(U_0,\psi_0)$, $(U_1,\psi_1)$ such that 
	\begin{equation}\label{eq:def_overalp}
		\psi_0(U\setminus\ell)=\psi_1(U\setminus\ell)=U_0\cap U_1.
	\end{equation}
	Assume that the change of coordinates map has the following expression in coordinates \eqref{eq:cyl_coords_U}
	\begin{equation}\label{eq:psi_def}
		\psi_0^{-1}\circ \psi_1:U\setminus\ell\to U\setminus \ell,\qquad \psi_0^{-1}\circ \psi_1(r,\theta,z)=(1-r,-q\theta+mz,p\theta+sz),
	\end{equation}
	where $m,s\in\mathbb Z$ and $mp+sq=1$. Then $M$ is diffeomorphic to $L(p,q)$.
\end{prop}
\begin{proof}
	We denote with $V\subset U$ the following set 
	\begin{equation}
		V=\left\{(re^{i\theta},z)\in U\,:\, r\leq 1/2\right\}\simeq D^2\times S^1.
	\end{equation}
	Since the $\psi_j$'s are charts, the following maps
	\begin{equation}
		\varphi_j=\psi_j|_{V}:V \to M,\qquad j=0,1,
	\end{equation}
	are smooth embeddings. We observe that 	
		\begin{equation}\label{eq:union_charts}
			M=\varphi_0(V)\cup\varphi_1(V).
	\end{equation}
	Indeed, fix $x\in M$. Then there exists $(re^{i\theta},z)\in U$ such that $x=\psi_j(re^{i\theta},z)$ for some $j\in\{0,1\}$. Assume for instance $j=1$. If $r\leq 1/2$ then $x\in \varphi_1(V)$. Otherwise, if $r> 1/2$ we have 
\begin{eqnarray}
	\psi_0^{-1}(x)=\psi_0^{-1}\circ\psi_1(re^{i\theta},z)=((1-r)e^{i(-q\theta+mz)},p\theta+sz)\in V.
\end{eqnarray}
Applying $\psi_0$ to both sides of the latter inclusion we find that $x\in \varphi_0(V)$. The case $j=0$ is analogous. Next, we observe that 
\begin{equation}\label{eq:intersection_charts}
	\varphi_0(V)\cap\varphi_1(V)=\varphi_0(\partial V)=\varphi_1(\partial V).
\end{equation}
Indeed, let $(r e^{i\theta},z),(r'e^{i\theta'},z')\in V$, then 
\begin{equation}
	\begin{aligned}
		\varphi_0(r e^{i\theta},z)=\varphi_1(r'e^{i\theta'},z')&\iff (r e^{i\theta},z)=\varphi_0^{-1}\varphi_1(r'e^{i\theta'},z')\\
		&\iff (r e^{i\theta},z)=((1-r')e^{i(-q\theta'+mz')},p\theta'+sz').
	\end{aligned}
\end{equation}
Since $r,r'\in[0,1/2]$ we find that the latter equation has the unique solution
\begin{equation}
	r=r'=1/2,\qquad (\theta,z)=(-q\theta'+mz',p\theta'+sz').
\end{equation}
Defining the map 
\begin{equation}
	\varphi:\partial V\to \partial V,\qquad \varphi(\theta,z)=(-q\theta+mz,p\theta+sz),
\end{equation}
we have proved that 
\begin{equation}\label{eq:F_well_defined}
	\varphi_0(x_0)=\varphi_1(x_1)\iff x_0,x_1\in \partial V\,\,\text{and}\,\, x_0=\varphi(x_1).
\end{equation}
Let $V_0,V_1$ be two copies of $V$.  By definition $L(p,q)=V_0\sqcup_{\varphi} V_1$. We define the map 
\begin{equation}
	F:V_0\sqcup_{\varphi} V_1\to M,\qquad F|_{V_j}=\varphi_j,\qquad j=0,1.
\end{equation}
Such map is well defined due to \eqref{eq:F_well_defined}. Moreover, from \eqref{eq:union_charts} and \eqref{eq:F_well_defined} we see that $F$ is bijective. In conclusions, $F$ is a bijection which restricts to the smooth embedding $\varphi_0$ on $V_0$, and to the smooth embedding $\varphi_1$ on $V_1$. Hence $F$ is a diffeomorphism.
\end{proof}
We exploit the toroidal atlas just introduced to define an action of the 2-torus $T^2$ on $L(p,q)$. Given a left action $\sigma: G\times X\to X$ of a group $G$ on a topological space $X$, we denote 
\begin{equation}
g\cdot x:=\sigma(g,x),\qquad g\in G, x\in X.
\end{equation} 
We identify the 2-torus with $(\mathbb R/2\pi\mathbb N)\times (\mathbb R/2\pi\mathbb N)$ and denote its elements with $(\theta',z')\in T^2$.
\begin{prop}\label{prop:T2_action}
	There exists a $T^2$-action on $L(p,q)$ satisfying
	\begin{equation}\label{eq:action0}
		(\theta',z')\cdot \psi_0(re^{i\theta},z)=\psi_0(re^{i(\theta+\theta')},z+z'),
	\end{equation}
for all $(\theta',z')\in T^2$ and all $(re^{i\theta},z)\in U$.
\end{prop}
\begin{proof}
Let $\sigma_0,\sigma_1:T^2\times U\to U$ denote the following actions of $T^2$ on $U$
\begin{equation}
	\begin{aligned}
	&\sigma_0((\theta',z'),(re^{i\theta},z))=\left(re^{i(\theta+\theta')},z+z'\right),\\
	&\sigma_1((\theta',z'),(re^{i\theta},z))=\left(re^{i(\theta-s\theta'+mz')},z+p\theta'+qz'\right).
	\end{aligned}
\end{equation}
We then define the action $\sigma:T^2\times L(p,q)\to L(p,q)$ as 
\begin{equation}
	\sigma(\cdot,\cdot)=\begin{cases}
		\psi_0\circ\sigma_0(\cdot,\psi_0^{-1}(\cdot)),\qquad \text{on $T^2\times U_0$},\\
		\psi_1\circ\sigma_1(\cdot,\psi_1^{-1}(\cdot)),\qquad \text{on $T^2\times U_1$.}
	\end{cases}
\end{equation}
To ensure that $\sigma$ is well defined on the overlap $T^2\times U_0\cap U_1$ we need to check that 
\begin{equation}
	\psi_0\circ\sigma_0(\cdot,\psi_0^{-1}(\cdot))=\psi_1\circ\sigma_1(\cdot,\psi_1^{-1}(\cdot)),\qquad \text{on $T^2\times U_0\cap U_1$},
\end{equation}
or, equivalently, we need to show that 
\begin{equation}\label{eq:well_def_mu}
	\sigma_1(\cdot,\cdot)=(\psi_0^{-1}\circ\psi_1)^{-1}\circ\sigma_0(\cdot,\psi_0^{-1}\circ\psi_1(\cdot)),\qquad \text{on $T^2\times U_0\cap U_1$}.
\end{equation}
According to the explicit expression of $\psi_0^{-1}\circ\psi_1$ in \eqref{eq:psi_def} we have 
\begin{equation}
	\begin{aligned}
		\psi_0^{-1}\circ \psi_1(re^{i\theta},z)&=((1-r)e^{i(-q\theta+mz)},p\theta+sz),\\ (\psi_0^{-1}\circ \psi_1)^{-1}(re^{i\theta},z)&=((1-r)e^{i(-s\theta+mz)},p\theta+qz).
	\end{aligned}
\end{equation}
Therefore we can directly check the validity of \eqref{eq:well_def_mu}:
\begin{equation}
	\begin{aligned}
		(\psi_0^{-1}\circ\psi_1)^{-1}\circ\sigma_0((\theta',z'),\psi_0^{-1}\circ\psi_1(re^{i\theta},z))&=(\psi_0^{-1}\circ\psi_1)^{-1}((1-r)e^{i(\theta'-q\theta+mz)},z'+p\theta+sz)\\
		&=(re^{i(-s\theta'+mz'+\theta)},p\theta'+qz'+z)=\sigma_1((\theta',z'),re^{i\theta}).
	\end{aligned}
\end{equation}
Thus $\sigma$ is well defined and, by definition, it satisfies \eqref{eq:action0}.
\end{proof}
\section{Tubular neighborhoods of Reeb orbits}\label{sec:tubes}
In this section we study the normal exponential map from a periodic orbit of the Reeb vector field in a K-contact Riemannian manifolds. In particular, we give an expression of the contact form in normal coordinates. We begin by recalling some standard facts concerning tubular neighborhoods of compact submanifolds in Riemannian geometry. 
\subsection{Tubes in Riemannian geometry }\label{ssec:tube_Riemann}
Let $(M,g)$ be a compact Riemannian manifold and let $\ell$ be a compact submanifold. The normal bundle of $\ell$ is the orthogonal complement of $T\ell$
\begin{equation}
	\nu(\ell)=(T\ell)^\perp.
\end{equation}
Given $s>0$ we define the following subsets of $\nu(\ell)$
\begin{equation}
	\nu^{<s}(\ell)=\{v\in \nu(\ell)\,:\, |v|<s\},\qquad \nu^s(\ell)=\{v\in \nu(\ell)\,:\, |v|=s\}.
\end{equation}
We define the normal exponential map of $\ell$ as 
\begin{equation}
	\exp^\perp:\nu(\ell)\to M,\qquad \exp^\perp(v)=\exp(v),
\end{equation}
where $\exp:TM\to M$ denotes the Riemannian exponential map. We denote with 
\begin{equation}
	\delta:M\to\mathbb R,\qquad \delta(x)=\inf_{y\in \ell} d(x,y),
\end{equation}
the distance from $\ell$ function, where $d(\cdot,\cdot)$ is the Riemannian distance. 
\begin{rmk}\label{rmk:geodesics_orthognal}
	Let $x\in M$ with $\delta(x)>0$. By compactness of $M$ and $\ell$ there exists an arc-length paramentrized geodesic $\sigma:[0,\delta(x)]\to M$ such that $\sigma(0)=0$ and $\sigma(\delta(x))=x$. It is a well know consequence of first order optimality conditions that $\dot\sigma(0)\in \nu^1(\ell)$. In particular \[\sigma(t)=\exp^\perp(t\dot\sigma(0)),\qquad \forall\,t\in [0,\delta(x)].\] See for instance \cite[Chap.\,III]{Sakai1996}.
\end{rmk}
\begin{defi}\label{def:cut_time}
	The cut time function is defined as 
	\begin{equation}
		t_{\cut}:\nu^1(\ell)\to \mathbb R,\qquad t_{\cut}(v)=\inf\{t>0\,:\, \delta(\exp^\perp(tv))=t\}.
	\end{equation}
	The injectivity domain is defined as 
	\begin{equation}
		U(\ell)=\left\{tv\in \nu(\ell)\,:\, v\in \nu^1(\ell),\,\,0\leq t< t_\cut(v)\right\}.
	\end{equation}
	The cut locus is defined as 
	\begin{equation}
		\mathrm{cut}(\ell)=\{\exp^\perp(t_{\cut}(v)v)\,:\, v\in \nu^1(\ell)\}.
	\end{equation}
\end{defi}
\noindent
The following theorem is well known and can be obtained from \cite[Chap.\,III]{Sakai1996}. In the case $\ell=\{pt.\}$ its statement can be found in most references on Riemannian geometry (e.g. \cite{DoCarmo}).
\begin{theorem}\label{thm:cut_locus}
	The cut time function is continuous and positive. In particular the injectivity domain $U(\ell)\subset \nu(\ell)$ is open, the cut locus $\cut(\ell)\subset M$ is closed and the map
	\begin{equation}
		\exp^\perp|_{U(\ell)}:U(\ell)\to M\setminus \cut(\ell),
	\end{equation} 
	is a diffeomorphism. 
\end{theorem}
\subsection{Tubes around Reeb orbits}\label{subsec:Tube_Reeb}
Throughout this subsection, $(M,\alpha,g)$ denotes a compact K-contact Riemannian manifold and $\ell$ a periodic orbit of minimal period $\tau>0$ of the Reeb field $R$. For a fixed $q\in\ell$, let $\gamma:\mathbb R/2\pi\mathbb N\to M$ be the following embedding of $\ell$:
\begin{equation}\label{eq:gamma}
	\gamma:\mathbb R/2\pi\mathbb N\to M,\qquad \gamma(t)=e^{\frac{\tau}{2\pi}tR}(q).
\end{equation}
By the compatibility condition \eqref{eq:comp_metric}, $\nu(\ell)$ coincides with $\xi|_{\ell}$. Hence, fixing an orthonormal frame
\begin{equation}\label{eq:adapted_frame}
	v_1,v_2\in\xi,\qquad g(v_i,v_j)=\delta_{ij},\qquad i,j=1,2,
\end{equation}
for $\xi|_\ell$, we may identify $\nu(\ell)$ with the trivial bundle $\mathbb R^2\times S^1$. We have the following proposition.
\begin{prop}\label{prop:bundle_isomorphism}
The following map is an isomorphism of vector bundles 
\begin{equation}
	v:\mathbb R^2\times S^1\to \nu(\ell),\qquad v(x,y,z)=(xv_1+yv_2)|_{\gamma(z)}.
\end{equation}
\end{prop}
\begin{remark}\label{rmk:cyl_coords}
	 The map $v:\mathbb R^2\times S^1\to \nu(\ell)$ of the previous proposition also induces cylindrical coordinates $(r,\theta,z)$ on $\nu(\ell)\setminus\ell$ defined by 
	 \begin{equation}
	 	x=r\cos\theta,\qquad y=r\sin\theta, \qquad z=z.
	 \end{equation}
	In particular $(\theta,z)$ are coordinates on the torus $\nu^s(\ell)$, for every $s>0$. 
\end{remark}
We now define the lift of the Reeb vector field to the tangent bundle.
\begin{defi}
We define the vector field $\vec R\in \mathrm{Vec}(TM)$ as 
\begin{equation}
	\vec R_{v}=\left.\frac{d}{dz}\right|_{z=0}e^{zR}_*v,\qquad \forall \, v\in TM.
\end{equation}
\end{defi}
Since the flow of the Reeb vector field acts by isometries and fixes $\ell$, $\vec R$ is tangent to $\nu(\ell)$. We can therefore express $\vec R$ in the cylindrical coordinates $(r,\theta,z)$ of \cref{rmk:cyl_coords}.
\begin{prop}\label{prop:Reeb_coords}
	There exists cylindrical coordinates $(r,\theta,z)$ on $\nu(\ell)\setminus\ell$, such that  
	\begin{equation}\label{eq:vec_R}
		\vec R=\frac{2\pi}{\tau}\left(\partial_z+\phi\partial_\theta\right),
	\end{equation}
where $\phi$ is the rotation number of the Reeb orbit $\ell$ of \cref{def:rotation_num}.
\end{prop}
\begin{proof}
	Let $v_1,v_2$ be an orthonormal frame for $\xi|_\ell$ of the form \eqref{eq:adapted_frame}. Let $\gamma:S^1\to M$ be the parametrization \eqref{eq:gamma} of $\ell$. Since $\alpha$ is K-contact, the map $e^{tR}_*=e^{t\vec R}$ rotates the frame $v_1,v_2$. Therefore there exists a smooth function $\varphi(t):\mathbb R\to\mathbb R$ such that 
	\begin{equation}\label{eq:flow_vecR}
		e^{t\vec R} v_1|_q=(\cos\varphi(t)v_1+\sin\varphi(t) v_2)|_{\gamma(\frac{2\pi}{\tau}t)},\qquad e^{t\vec R} v_2|_q=(-\sin\varphi(t)v_1+\cos\varphi(t) v_2)|_{\gamma(\frac{2\pi}{\tau}t)}.
	\end{equation}
	Observe that, by \cref{def:rotation_num} of $\phi$, the function $\varphi$ satisfies 
	\begin{equation}\label{eq:boundaryvarphi}
		\varphi(0)\equiv 0\mod 2\pi,\qquad \varphi(\tau)\equiv 2\pi\phi \mod 2\pi.
	\end{equation}
	Thanks to \eqref{eq:flow_vecR}, we can compute the vector field $\vec R$ in cylindrical coordinates of \cref{rmk:cyl_coords}
	\begin{equation}
		\begin{aligned}
			e^{t\vec R}v(r\cos\theta,r\sin\theta,z)&=e^{t\vec R}(r\cos\theta v_1+r\sin\theta v_2)|_{\gamma(z)}\\
			&=v\left(r\cos\left(\theta+\varphi(t)\right),r\sin\left(\theta+\varphi(t)\right),z+\frac{2\pi}{\tau}t\right),
		\end{aligned}
	\end{equation}
	where $v$ is the map of \cref{prop:bundle_isomorphism}. Differentiating both sides with respect to $t$ we find 
	\begin{equation}\label{eq:preReeb}
		\vec R=v_*\left(\frac{2\pi}{\tau}\partial_z+\dot\varphi\left(\frac{\tau}{2\pi}z\right)\partial_\theta\right).
	\end{equation}
	Consider the following change of coordinates 
	\begin{equation}\label{eq:change_of_coords}
		(\theta,z)\mapsto (\theta',z'):=\left(\theta+\frac{2\pi}{\tau}\left(z\phi-\varphi\left(\frac{\tau}{2\pi}z\right)\right),z\right).
	\end{equation}
	Note that \eqref{eq:change_of_coords} is is well defined since, by \eqref{eq:boundaryvarphi}, $(\theta'(\theta,2\pi),z'(\theta,2\pi))=(\theta'(\theta,0),z'(\theta,0)+2\pi)$. Applying the change of coordinates \eqref{eq:change_of_coords} to \eqref{eq:preReeb}, we find 
	\begin{equation}
		\vec R=v_*\left(\frac{2\pi}{\tau}\left(\partial_{z'}+\phi\partial_{\theta'}\right)\right),
	\end{equation}
	concluding the proof.
\end{proof}
\begin{remark}
	Let $m\in\mathbb Z$, applying the change of coordinates $(\theta,z)\mapsto (\theta+mz,z)$ to \eqref{eq:vec_R} yields 
	\begin{equation}
		\vec R=\frac{2\pi}{\tau}\left(\partial_z+(\phi+m)\partial_\theta\right).
	\end{equation}
	We see that such transformation has the same effect as changing the representative of $\phi$ mod $1$.
\end{remark}
\begin{remark}\label{rmk:commutativity_E}
	Since the flow of $R$ acts by isometries and fixes $\ell$, the following holds 
	\begin{equation}
		e^{tR}\circ \exp^\perp=\exp^\perp\circ e^{t\vec R},\qquad \forall \,t\in\mathbb R.
	\end{equation}
\end{remark}
\begin{remark}\label{rmk:dense_on_Nasell}
	If $\alpha$ is irregular, then the orbits of $\vec R$ are dense on $\nu^s(\ell)$ for all $s>0$. Indeed, in this case at least one non-periodic Reeb orbit exists, thus by \cref{lem:rot_determines_flow} the rotation number $\phi$ is irrational. Therefore, the explicit expression of $\vec R$ in \cref{prop:Reeb_coords} shows that, for any $s>0$, the orbits of the flow of $\vec R$ are dense on the torus  
	\begin{equation}
		\nu^s(\ell)=\{v\in \nu(\ell)\,:\,|v|=s\}\simeq \{(r\cos\theta,r\sin\theta,z)\in \mathbb R^2\times S^1\,:\,r=s\}.
	\end{equation}
\end{remark}
We now compute the form $(\exp^{\perp*}\alpha)$ in the coordinates of \cref{prop:Reeb_coords}.
\begin{lemma}\label{lem:geodesics_and_form}
Let $\exp^\perp:\nu(\ell)\to M$ be the normal exponential map, then there exists a smooth function $a:\nu(\ell)\to \mathbb R$, such that
	\begin{equation}\label{eq:E*alpha}
		(\exp^{\perp*}\alpha)=\left(\frac{\tau}{2\pi}-\phi a\right)dz+ad\theta.
	\end{equation}
Furthermore, if $\alpha$ is irregular then the following equalities hold on the whole $\nu(\ell)$
\begin{equation}\label{eq:theta_z_invariance}
	\partial_\theta a=\partial_z a=0,\qquad \mathcal L_{\partial_\theta}\left(\exp^{\perp*}g\right)=\mathcal  L_{\partial_z}\left(\exp^{\perp*}g\right)=0.
\end{equation}
\end{lemma}
\begin{proof}
We claim that $\alpha(\exp^\perp_*\partial_r)\equiv 0$. Let $v\in \nu^1\ell$ and consider the geodesic
\begin{equation}
		\sigma:[0,+\infty)\to M,\qquad \sigma(t)=\exp^\perp(tv).
\end{equation} 
Note that $\exp^\perp_*\partial_r=\dot\sigma$. We define $u:[0,+\infty)\to \mathbb R$ as
\begin{equation}
	u(t)=\alpha(\dot\sigma(t)).
\end{equation}
To prove the claim, we need to show that $u\equiv 0$. Let $X$ be a local unit norm vector field extension of $\dot\sigma$. Note that, since $R$ is Killing it holds 
\begin{equation}\label{eq:computation_nabla0}
	\begin{aligned}
		0=\left(\mathcal L_Rg\right)(X,X)=R(g(X,X))-g([R,X],X)-g(X,[R,X])=2g(X,\nabla_XR),
	\end{aligned}
\end{equation}
where $\nabla$ is the Levi-Civita connection. It follows from \eqref{eq:comp_metric} that $u(t)=g(\dot\sigma(t),R)$, therefore 
\begin{equation}	
	\dot u(t)=g(\nabla_{\dot\sigma}\dot\sigma(t),R)+g(X,\nabla_XR)=0,\qquad \forall\, t\geq 0,
\end{equation}
where the first addendum is zero since $\sigma$ is a geodesic, the second one because of \eqref{eq:computation_nabla0}. Since $u(0)=0$, because $\dot\sigma(0)\in \nu(\ell)=\xi|_\ell$, the claim is proved. Therefore it holds 
\begin{equation}\label{eq:E*alpha_proof}
	(\exp^{\perp*}\alpha)=(\exp^{\perp*}\alpha)(\partial_\theta)d\theta+(\exp^{\perp*}\alpha)(\partial_z)dz.
\end{equation}
Furthermore, thanks to \cref{rmk:commutativity_E} and \cref{prop:Reeb_coords} we can compute 
\begin{equation}
	\begin{aligned}
		1=\alpha(R)=\alpha(\exp^\perp_*\vec R)&=(\exp^{\perp*}\alpha)(\vec R)=\frac{2\pi}{\tau}(\exp^{\perp*}\alpha)\left(\partial_z+\phi\partial_\theta\right).
	\end{aligned}
\end{equation}
Setting $a:=(\exp^{\perp*}\alpha)(\partial_\theta)$, the above equation requires that $$(\exp^{\perp*}\alpha)(\partial_z)=\frac{\tau}{2\pi}-\phi a,$$ which substituted in \eqref{eq:E*alpha_proof} confirms the expression of $(\exp^{\perp*}\alpha)$ in the statement. 
\newline\newline\noindent
If $\alpha$ is irregular then by \cref{rmk:dense_on_Nasell} any orbit of $\vec{R}$ is dense on the torus $\nu^s(\ell)$, for all $s>0$. Moreover, according to \cref{rmk:cyl_coords}, $T(\nu^s(\ell))=\mathrm{span}\{\partial_\theta,\partial_z\}$. Therefore, for any smooth function $f:\nu(\ell)\to \mathbb R$ the following equivalence holds 
\begin{equation}
	\vec R(f)=0\iff \partial_\theta f=\partial_z f=0.
\end{equation}
Similarly, for any tensor smooth tensor $T$ defined on $\nu(\ell)$ one has
\begin{equation}\label{eq:tensor_invariance}
	\mathcal L_{\vec R}T=0\iff \mathcal L_{\partial_\theta }T=\mathcal L_{\partial_z}T=0.
\end{equation}
Since $\mathcal L_R g=0$, by \cref{rmk:commutativity_E} it also holds $\mathcal L_{\vec R}(\exp^{\perp*}g)=0$, indeed 
\begin{equation}
	0=\exp^{\perp*}(\mathcal L_R g)=\mathcal L_{\exp^\perp_*R}(\exp^{\perp*}g)=\mathcal L_{\vec R}(\exp^{\perp*}g).
\end{equation}
Therefore by \eqref{eq:tensor_invariance} it holds $$\mathcal L_{\partial_\theta }(\exp^{\perp*}g)=\mathcal L_{\partial_z}(\exp^{\perp*}g)=0.$$ 
Similarly, since $\mathcal L_{R}\alpha=0$ one also has $\mathcal L_{\vec R}(\exp^{\perp*}\alpha)=0$. On the other hand, using the explicit expression \eqref{eq:vec_R} of $\vec R$ and the one \eqref{eq:E*alpha} of $(\exp^{\perp*}\alpha)$  yields 
\begin{equation}
	0=\mathcal L_{\vec R}(\exp^{\perp*}\alpha)=\vec R(a)(d\theta-\phi dz),
\end{equation}
which implies $\vec R(a)=0$ and, by \eqref{eq:tensor_invariance}, $\partial_\theta a=\partial_z a=0$.
\end{proof}
\section{Characterization of irregular K-contact forms}\label{sec:char_K_cont}
The purpose of this section is to prove the following result. 
\begin{theorem}\label{thm:gather_results}
	Let $(M,\alpha)$ be a closed K-contact manifold with $\alpha$ irregular. Then: 
	\begin{itemize}
		\item[(i)] there exists exactly two periodic Reeb orbits $\ell_0$, $\ell_1$ of minimal periods $\tau_0,\tau_1$ with $\tau_0/\tau_1$ irrational,
		\item[(ii)] there exists $p,q\in\mathbb N$ co-prime such that $M$ is diffeomorphic to $L(p,q)$,
		\item[(iii)] let $\phi_0$ be the rotation number of $\ell_0$, then there exists a smooth function $a:[0,1]\to\mathbb R$ such that 
		\begin{equation}
			\psi_0^*\alpha=\left(\frac{\tau_0}{2\pi}-\phi_0 a(r)\right)dz+a(r)d\theta,\qquad r\in [0,1],
		\end{equation}
		where $(U_0,\psi_0)$ is one of the two toroidal charts of \cref{prop:charts}.
	\end{itemize}
\end{theorem}
The proof is divided into three lemmas, each of which is proved in a separate subsection. Throughout the section $(M,\alpha)$ denotes a smooth compact K-contact manifold. By \cref{thm:one_periodic_orbit} there exists at least one periodic Reeb orbit $\ell\subset M$. Let $\tau>0$ be its minimal period. We fix a K-contact Riemannian metric $g$ compatible with $\alpha$, so that $(M,\alpha,g)$ is a K-contact Riemannian manifold, and referring to the notation of \cref{sec:tubes}, we introduce the normal exponential map of $\ell$
\begin{equation}\label{eq:normal_recall}
	\exp^\perp:\nu(\ell)\to M.
\end{equation}
We denote with $\delta:M\to\mathbb R$ the distance from $\ell$, i.e., for $x\in M$ 
\begin{equation}\label{eq:def_delta}
	\delta(x)=\min_{y\in\ell}d(y,x).
\end{equation}
\subsection{Cut time and non existence of periodic orbits}
In this subsection we compute the cut time (cf. \cref{def:cut_time}) of the exponential map \eqref{eq:normal_recall}. We then use the result of such computation to show that there are no periodic Reeb orbits away from $\ell$ and from the cut locus. In the following lemma, for $s>0$ we denote 
\begin{equation}\label{eq:ball_centred_at_ell}
	B_s(\ell)=\{x\in M\,:\, \delta(x)<s\}.
\end{equation}

	\begin{lemma}\label{lem:constant_t_cut}
		Let  $\rho=\max_{x\in M} \delta(x)$, then the cut time function is constant and identically equal to $\rho$. Furthermore, the Reeb vector field does not have any periodic orbit on $B_\rho(\ell)\setminus \ell$.
	\end{lemma}
	\begin{proof}
		The flow of the Reeb field $R$ acts by isometries and fixes $\ell$. Therefore the distance from $\ell$ function, which is $\delta:M\to \mathbb R$, is invariant under the flow of $R$. The latter fact, combined with \cref{rmk:commutativity_E}, implies that $t_\cut:\nu^1(\ell)\to \mathbb R$ is constant along the orbits of $\vec R$, indeed 
		\begin{equation}
			\begin{aligned}
				t_\cut(e^{t\vec R}v)&=\inf\{r>0\,:\, \delta(\exp^\perp\circ e^{t\vec R}(rv))=r\}=\inf\{r>0\,:\, \delta(e^{tR}\circ \exp^\perp(rv))=r\}\\
				&=\inf\{r>0\,:\, \delta(\exp^\perp(rv))=r\}=t_\cut(v),
			\end{aligned}
		\end{equation}
		for any $v\in \nu^1(\ell)$ and $t\in\mathbb R$. Since, by \cref{rmk:dense_on_Nasell}, every orbit of $\vec R$ is dense in $\nu^1(\ell)$ and $t_\cut$ is continuous (cf. \cref{thm:cut_locus}), we deduce that $t_\cut$ is constant. By definition of $t_{\cut}$ and \cref{rmk:geodesics_orthognal}, the maximal value of $t_\cut$ equals the maximal value of $\delta$, therefore $t_\cut\equiv \rho$. Thus by \cref{thm:cut_locus} the map 
		\begin{equation}\label{eq:expperp_rho}
			\exp^\perp|_{\nu^{<\rho}(\ell)}:\nu^{<\rho}(\ell)\to B_\rho(\ell)
		\end{equation}
		is a diffeomorphism. Hence, given $x\in B_\rho(\ell)\setminus \ell$ there exists a unique $v\in \nu^{<\rho}(\ell)\setminus\ell$ such that $x=\exp^\perp(v).$ Let $t\in\mathbb R$ such that $e^{tR}(x)=x$. We show that $t=0$. By \cref{rmk:commutativity_E} we have 
		\begin{equation}
			\exp^\perp\circ e^{t\vec R}(v)=e^{tR}\circ \exp^\perp(v)=e^{tR}(x)=x=\exp^\perp(v).
		\end{equation}
		Both $v$ and $e^{t\vec R}(v)$ belong to $\nu^{<\rho}(\ell)\setminus \ell$, where $\ell\subset \nu(\ell)$ is identified with the zero section. By injectivity of \eqref{eq:expperp_rho} we deduce $v=e^{t\vec R}(v)$. 
		However \cref{rmk:dense_on_Nasell} implies that $\vec R$ has no periodic orbit on $\nu(\ell)\setminus \ell$, therefore $t=0$. 
	\end{proof}
	\subsection{Volume function and cut locus}\label{ssec:vol_fun}
	We prove that the following function
	\begin{equation}\label{eq:def_mu}
		\mu:M\to \mathbb R,\qquad \mu(x)=\mathrm{vol}\left(B_{\delta(x)}(\ell)\right)=\int_{B_{\delta(x)}(\ell)}\alpha\wedge d\alpha,
	\end{equation}
	where $\delta(x)$ is defined in \eqref{eq:def_delta} and $B_{r}(\ell)$ in \eqref{eq:ball_centred_at_ell}, is a smooth K-contact momentum map. We then combine the latter fact with \cref{thm:momentum_maps} to show that $\cut(\ell)$ is a periodic Reeb orbit. 

	In the following, for $r>0$ we denote $S_r(\ell)=\{x\in M\,:\, \delta(x)=r\}$.
	\begin{lemma}\label{lem:smooth_vol}
		The function \eqref{eq:def_mu} is a smooth K-contact momentum map, moreover
		\begin{equation}\label{eq:E*mu=a}
			\mu\circ \exp^\perp=2\pi\tau(\exp^{\perp*}\alpha)(\partial_\theta),\quad \text{on $\nu^{<\rho}(\ell)$.}
		\end{equation}
	Furthermore, the cut locus $\cut(\ell)$ is a periodic Reeb orbit. In particular $M$ contains exactly two periodic Reeb orbits, namely $\ell$ and $\cut(\ell)$. 
	\end{lemma}
\begin{proof}
	 Since $t_\cut\equiv \rho$, by \cref{thm:cut_locus} for any $0<r<\rho$ the map
	\begin{equation}\label{eq:isometry_of_tori}
		\exp^\perp|_{\nu^{r}(\ell)}:\nu^{r}(\ell)\to S_r(\ell),
	\end{equation}
is a diffeomorphism and induces an isometry of Riemannian manifolds, the metrics being \[(\exp^{\perp*}g)|_{\nu^r(\ell)}\quad \text{and}\quad g|_{S_r(\ell)},\] on the domain and on target, respectively.
Since \eqref{eq:isometry_of_tori} is an isometry, and since by \cref{lem:geodesics_and_form} $\partial_\theta$ is a Killing vector field for $\nu^r(\ell)$, then $\exp^\perp_*\partial_\theta$ is a Killing field for $S_r(\ell)$. Note that $S_r(\ell)$ is an embedded surface invariant under the Reeb flow and, by \cref{lem:constant_t_cut}, it contains no periodic Reeb orbit. Therefore, by \cref{rmk:Killing_extension}, $\exp^\perp_*\partial_\theta$ extends to a global contact Killing vector field on $M$. We denote with $X\in\mathrm{Vec}(M)$ the extension of $\exp^\perp_*\partial_\theta$ and define the smooth map
\begin{equation}
	f:M\to\mathbb R,\qquad f(x)=\alpha(X)|_{x}.
\end{equation}
Note that, since $X$ is a K-contact field, then $f$ is a K-contact momentum map, indeed 
\begin{equation}
	df=d(\alpha(X))=\mathcal L_X\alpha-\iota_Xd\alpha=-\iota_X d\alpha.
\end{equation}
According to \cref{lem:geodesics_and_form}, there exists a smooth function $a:\nu(\ell)\to\mathbb R$, such that 
\begin{equation}\label{eq:expression_of_alpha2}
	(\exp^{\perp*}\alpha)=\left(\frac{\tau}{2\pi}-\phi a\right)dz+ad\theta.
\end{equation}
Since $X=\exp^\perp_*\partial_\theta$ and $a=(\exp^{\perp*}\alpha)(\partial_\theta)$, we deduce that 
\begin{equation}\label{eq:f_and_h}
	f\circ \exp^\perp=a\quad \text{on $\nu^{<\rho}(\ell)$}.
\end{equation}
By \cref{lem:geodesics_and_form} the function $a$ is radial, therefore \eqref{eq:f_and_h} rewrites in cylindrical coordinates as
\begin{equation}
	f\circ \exp^\perp(r\cos\theta, r\sin\theta,z)=a(r),\qquad 0<r<\rho.
\end{equation}
Note that, since $\partial_\theta|_{r=0}=(x\partial_y-y\partial_x)|_{r=0}=0$, it holds $$a(0)=0.$$
Let $x\in M$ such that $\delta(x)<\rho$. Then $x\in \exp^\perp(\nu^{<\rho}(\ell))$, hence we can write $$x=\exp^\perp(\delta(x)\cos\theta,\delta(x)\sin\theta,z)$$ for some $\theta,z\in \mathbb R/2\pi\mathbb N$. Exploiting the fact that $\exp^\perp:\nu^{<\delta(x)}(\ell)\to B_{\delta(x)}(\ell)$ is a diffeomorphism and \eqref{eq:expression_of_alpha2}, we may compute 
\begin{equation}
	\begin{aligned}
		\mu(x)&=\int_{B_{\delta(x)}(\ell)}\alpha\wedge d\alpha=\int_{\nu^{<\delta(x)}(\ell)}(\exp^{\perp*}\alpha)(\alpha\wedge d\alpha)=\int_{0}^{2\pi}\int_{0}^{2\pi}\int_0^{\delta(x)}\frac{\tau}{2\pi}\partial_r a(r)drd\theta dz\\
		&=2\pi\tau a(\delta(x))=2\pi\tau 	f(\exp^\perp(\delta(x)\cos\theta,\delta(x)\sin\theta,z))=2\pi\tau f(x).
	\end{aligned}
\end{equation} 
We have obtained the following equality 
\begin{equation}\label{eq:mu_and_f}
	\mu=2\pi\tau f,\qquad \text{on $B_{\rho}(\ell)\subset M$}.
\end{equation}
Since $\rho=\max_{x\in M}\delta(x)$, the set $B_{\rho}(\ell)$ is dense in $M$. We now observe that $\mu$ is continuous. To see this, we write $\mu=\mathcal V\circ \delta$, where 
\begin{equation}
	\mathcal V:[0,+\infty)\to \mathbb R,\qquad \mathcal V(t)=\int_{\{\delta<t\}}\mathrm{vol}_g.
\end{equation}
Since $\delta$ is clearly continuous, we need to show that $\mathcal V$ is continuous. Note that 
\begin{equation}
	\lim_{\ve\to 0}|\mathcal V(t+\ve)-\mathcal V(t)|=\int_{\{\delta=t\}}\mathrm{vol}_g.
\end{equation}
By \cref{rmk:geodesics_orthognal} the following inclusion holds 
\begin{equation}
	\{\delta=t\}\subset\exp^\perp(\nu^t(\ell)).
\end{equation}
We see that $\{\delta=t\}$ is contained in the image, through the smooth map $\exp^\perp$, of the measure zero set $\nu^t(\ell)\subset \nu(\ell)$. Therefore $\{\delta=t\}$ is a negligible set and $\mathcal V$ is continuous. Thus by \eqref{eq:mu_and_f}, $\mu$ and $2\pi\tau f$ are two continuous functions coinciding on a dense subset of $M$. Therefore they coincide on the whole $M$.  Since $f$ is a smooth K-contact momentum map, the same holds for $\mu$. The validity of \eqref{eq:E*mu=a} is deduced from \eqref{eq:f_and_h}.

 Recall that, by \cref{lem:constant_t_cut}, $\cut(\ell)=S_\rho(\ell)$. Note that any point in $S_\rho(\ell)$ is necessarily a critical point of $\mu$, because $$\mu(x)\leq \mu|_{S_\rho(\ell)}\equiv \mathrm{vol}(M),\qquad \forall\,x\in M,$$ since $\rho=\max_{x\in M}\delta(x)$. Therefore, by \cref{thm:momentum_maps}, $\cut(\ell)$ is a union of periodic Reeb orbits. We now show that $\cut(\ell)$ is actually a single one. Consider the map 
\begin{equation}
	\zeta:\nu^\rho(\ell)\to \cut(\ell),\qquad \zeta=\exp^\perp|_{\nu^{\rho}(\ell)}.
\end{equation}
Note that $\zeta$ is a continuous surjective map. Fix $v\in \nu^{\rho}(\ell)$, let $\tilde \ell_1$ be the orbit of $\vec R$ containing $v$, and let $\ell_1$ be the orbit of the Reeb field $R$ containing $\zeta(v)$. By \cref{rmk:commutativity_E}, $\ell_1=\zeta(\tilde \ell_1)$. Moreover, $\tilde \ell_1$ is dense in $\nu^{\rho}(\ell)$, by \cref{rmk:dense_on_Nasell}. Since $\zeta$ is continuous and surjective, $\ell_1$ is dense in $\cut(\ell)$. But $\ell_1$ is closed, since $\ell_1\subset \cut(\ell)$ which is contained in the critical set of $\mu$ which by \cref{thm:momentum_maps} is a union of periodic Reeb orbits. Therefore $\ell_1=\cut(\ell)$.
\end{proof}
\begin{remark}\label{rmk:mu_01}
	Let $\ell_0=\ell$ and $\ell_1=\cut(\ell)$. Note that the following equalities hold 
	\begin{equation}
		\mu|_{\ell_0}=0,\qquad \mu|_{\ell_1}=\mathrm{vol}(M).
	\end{equation}
 Thus the critical set of $\mu$ contains the union of $\ell_0$ and $\ell_1$: $\mu$ achieves its minimum value at $\ell_0$ and its maximum at $\ell_1$. Moreover, since $\ell_0,\ell_1$ are the only two periodic Reeb orbits, by \cref{thm:momentum_maps} $\mu$ has no critical point on $M\setminus \ell_0\cup\ell_1$.
\end{remark}
	\subsection{Rotation numbers and toroidal charts}
	We denote with $\ell_0,\ell_1\subset M$ the two periodic Reeb orbits of \cref{lem:smooth_vol}, in particular $\ell_0=\ell$ and $\ell_1=\cut(\ell_0)$. For $i=0,1$ we denote with 
	\begin{equation}
		\exp^\perp_i:\nu(\ell_i)\to M,
	\end{equation}
	the corresponding normal exponential maps. By \cref{lem:constant_t_cut} and \cref{lem:smooth_vol} it holds
	\begin{equation}
		\cut(\ell_0)=\ell_1,\qquad \cut(\ell_1)=\ell_0,\qquad t_\cut(\ell_0)=t_\cut(\ell_1)=\rho.
	\end{equation}
	Therefore by \cref{thm:cut_locus} the restrictions of the normal exponential maps yields diffeomorphisms 
	\begin{equation}
			\exp^\perp_i:\nu^{<\rho}(\ell_i)\to B_\rho(\ell_i),\qquad i=0,1.
	\end{equation}
	In this subsection we compute, in suitable coordinates, the transition map 
	\begin{equation}
		(\exp^\perp_0)\circ (\exp^\perp_1)^{-1}:B_\rho(\ell_0)\cap B_\rho(\ell_1)\to B_\rho(\ell_0)\cap B_\rho(\ell_1).
	\end{equation}
	In the process we also compute the rotation numbers $\phi_0,\phi_1$ of the Reeb orbits $\ell_0,\ell_1$.
	\begin{lemma}\label{lem:Phi_and_phis}
		There exists a diffeomorphism $\Phi:\nu^1(\ell_1)\to \nu^1(\ell_0)$ which satisfies
		\begin{equation}\label{eq:Phi_statement}
			\exp^\perp_0((\rho-r)\Phi(v))=\exp^\perp_1(rv),\qquad \forall\,\,v\in \nu^1(\ell_1),\quad \forall\,\,r\in[0,\rho].
		\end{equation} 
		Moreover, in coordinates $(\theta,z)$ of \cref{rmk:cyl_coords}, $\Phi$ can be written as
		\begin{equation}
			\Phi(\theta,z)=(-q\theta+mz,p\theta+sz),
		\end{equation}
		where $p,q,m,s\in\mathbb Z$ satisfy $pm+qs=1$. Furthermore, it holds 
		\begin{equation}\label{eq:vol_an_phi}
			\mathrm{vol}(M)=p\tau_0\tau_1,\qquad \phi_0=\frac{1}{p}\left(\frac{\tau_0}{\tau_1}-q\right),\qquad \phi_1=\frac{1}{p}\left(\frac{\tau_1}{\tau_0}-s\right),
		\end{equation} 
		where $\phi_0,\phi_1$ are the rotation numbers of $\ell_0,\ell_1$. In particular $p$ is positive.
	\end{lemma}
	\begin{proof}
		We define $\Phi:\nu^1(\ell_1)\to \nu^1(\ell_0)$ as 
		\begin{equation}
			\Phi(v)=-\left.\frac{d}{dr}\right|_{r=\rho}\exp^\perp_1(rv),\qquad v\in \nu^1(\ell_1).
		\end{equation}
		Note that $\Phi$ is well defined since for any $v\in \nu^1(\ell_0)$ the curve $[0,\rho]\ni t\mapsto \exp^\perp_0(tv)$ is an arc-length parametrized geodesic whose length equals the distance between $\ell_0$ and $\ell_1$, therefore its velocity at the endpoints has norm $1$ and is orthogonal to $\ell_0$ and $\ell_1$ respectively (cf. \cref{rmk:geodesics_orthognal}). Moreover $\Phi$ is a diffeomorphism, its inverse being
		\begin{equation}
			\Phi^{-1}:\nu^1(\ell_0)\to \nu^1(\ell_1),\qquad \Phi^{-1}(v)=-\left.\frac{d}{dr}\right|_{r=\rho}\exp^\perp_0(rv).
		\end{equation}
		We claim that for any $v\in \nu^1(\ell_0)$ and $r\in[0,\rho]$ it holds
		\begin{equation}\label{eq:what_Phi_is_needed_for}
			\exp^\perp_0(r\Phi(v))=\exp^\perp_1((\rho-r)v).
		\end{equation}
		Indeed, consider the curve $\sigma:[0,\rho]\to M$, defined by $$\sigma(r)=\exp^\perp_1((\rho-r)v).$$ Such curve is a minimizing geodesic joining $\ell_0$ to $\ell_1$ whose initial velocity is 
		\begin{equation}
			\dot\sigma(0)=\left.\frac{d}{dr}\right|_{r=0}\exp^\perp_1((\rho-r)v)=-\left.\frac{d}{dr}\right|_{r=\rho}\exp^\perp_1(rv)=\Phi(v).
		\end{equation}
		It follows that $\sigma(r)=\exp^\perp_0(r\Phi(v))$. We make the identification $\nu^1(\ell_i)=\partial D^2\times S^1$ and compute $\Phi$ in cylindrical coordinate of \cref{rmk:cyl_coords}. 
		From \cref{rmk:commutativity_E} and \eqref{eq:what_Phi_is_needed_for} we deduce that
		\begin{equation}
			\begin{aligned}
				\Phi(e^{t\vec R_1}v)=e^{t\vec R_0}\Phi(v),\qquad t\in\mathbb R, v\in \nu^1(\ell_1),
			\end{aligned}
		\end{equation}
		where $\vec R_i$ denotes the restriction of $\vec R$ to $\nu(\ell_i)$. Therefore we have $\Phi_*\vec R_1=\vec R_0$. Let $p,q:\nu^1(\ell_0)\to \mathbb R$ be two smooth functions such that 
		\begin{equation}\label{eq:diff_Phi_theta}
			\Phi_*\partial_\theta=p\partial_z-q\partial_\theta.
		\end{equation}
	  	From the explicit expression of $\vec R_i$ found in \cref{prop:Reeb_coords} we obtain $[\vec R_0,\partial_\theta]=0$. Therefore
		\begin{equation}
			0=\Phi_*[\vec R_1,\partial_\theta]=[\Phi_*\vec R_1,\Phi_*\partial_\theta]=[\vec R_0,p\partial_z-q\partial_\theta]=\vec R_0(p)\partial_z-\vec R_0(q)\partial_\theta.
		\end{equation}
		Since the orbits of $\vec R_0$ are dense on $\nu^1(\ell_1)$ (cf. \cref{rmk:dense_on_Nasell}), we deduce that $p,q$ are constant. Furthermore, since $\Phi$ is a diffeomorphism, any orbit of $\Phi_*\partial_\theta$ must be periodic of period $2\pi$, therefore $p$ and $q$ are coprime integers. Similarly, we deduce the existence of $m,s\in\mathbb Z$ such that 
		\begin{equation}\label{eq:diff_Phi_z}
			\Phi_*\partial_z=s\partial_z+m\partial_\theta.
		\end{equation}
		The integral curves of $\Phi_*\partial_\theta,\Phi_*\partial z$ must be generators of $H_2(\nu^1(\ell_0))=\mathbb Z^2$, thus by \eqref{eq:diff_Phi_theta},\eqref{eq:diff_Phi_z}
		\begin{equation}
			\pm 1=\det\begin{pmatrix}
				-q & m \\
				p & s	
			\end{pmatrix}=-mp-sq.	
		\end{equation}
		 Moreover, since for $i=0,1,$ $\exp_i^{\perp}$ is an orientation preserving diffeomorphism onto its image when restricted to $\nu^{<\rho}(\ell_i)$, \eqref{eq:what_Phi_is_needed_for} implies that $\Phi$ is orientation reversing, therefore 
		\begin{equation}\label{eq:generator_homology2}
			-mp-sq=-1.
		\end{equation}
		The two equation \eqref{eq:diff_Phi_theta} and \eqref{eq:diff_Phi_z} uniquely determine the map $\Phi$ up to a translation, in particular
		\begin{equation}
			\Phi(\theta,z)=(-q\theta+mz,p\theta+sz).
		\end{equation}
		%\begin{equation}\label{eq:diff_Phi_z}
		%	\Phi_*\partial_z=\left(\frac{\tau_1}{\tau_0}-\phi_1p\right)\partial_z+\left(\frac{\tau_1}{\tau_0}\phi_0+q\phi_1\right)\partial_\theta
		%\end{equation} 
		Note that $\Phi$ can be extended to a smooth map $\Phi:\nu(\ell_1)\setminus\ell_1\to\nu(\ell_0)\setminus\ell_0$ setting
		\begin{equation}
			\Phi(rv)=r\Phi(v),\qquad \forall\,v\in \nu^1(\ell_1), r>0.
		\end{equation}
		With such extension intended, \eqref{eq:what_Phi_is_needed_for} implies that 
		\begin{equation}
			\Phi^*(\exp_0^{\perp*}\alpha)|_{\nu^r(\ell_0)}=(\exp_1^{\perp*}\alpha)|_{\nu^{\rho-r}(\ell_1)},\qquad r\in(0,\rho).
		\end{equation}
		Substituting the expressions of $(\exp^\perp_i)^*\alpha$ found in \cref{lem:geodesics_and_form} in the above equality we obtain 
		\begin{equation}\label{eq:Phi*alpha}
			\Phi^*\left(\left(\frac{\tau_0}{2\pi}-\phi_0a_0(r)\right)dz+a_0(r)d\theta\right)=\left(\frac{\tau_1}{2\pi}-\phi_1a_1(\rho-r)\right)dz+a_1(\rho-r)d\theta.
		\end{equation}
		On the other hand from \eqref{eq:diff_Phi_theta} and \eqref{eq:diff_Phi_z} we deduce that
		\begin{equation}\label{eq:Phi_pullback}
			\Phi^*dz=sdz+pd\theta,\qquad \Phi^*d\theta=-qd\theta+mdz.
		\end{equation}
		Substituting the latter relations in \eqref{eq:Phi*alpha}, we deduce that, for $r\in (0,\rho)$ it holds
		\begin{equation}\label{eq:a_0_a_1_forms}
			\frac{s\tau_0}{2\pi}+a_0(r)(m-s\phi_0)=\frac{\tau_1}{2\pi}-\phi_1a_1(\rho-r),\qquad \frac{p\tau_0}{2\pi}-a_0(r)(q+p\phi_0)=a_1(\rho-r).
		\end{equation}
	     Since $a_i=(\exp_i^\perp)^*\alpha(\partial_\theta)$, according to \cref{lem:smooth_vol} it holds $2\pi\tau_ia_i=\mu_i\circ\exp^{\perp}_i$, where $\mu_i$ is the volume function of \cref{ssec:vol_fun} associated to the Reeb orbit $\ell_i$, $i=0,1$. Then, by \cref{rmk:mu_01}  
	     \begin{equation}\label{eq:values_ai}
	     	a_i(0)=0,\qquad a_i(\rho)=\frac{\mathrm{vol}(M)}{2\pi\tau_i},\qquad i=0,1.
	     \end{equation}
      Thus, taking the limit as $r\to 0$ of \eqref{eq:a_0_a_1_forms}, thanks to \eqref{eq:values_ai} we find
      \begin{equation}
      		\frac{s\tau_0}{2\pi}=\frac{\tau_1}{2\pi}-\phi_1\frac{\mathrm{vol}(M)}{2\pi\tau_1},\qquad \frac{p\tau_0}{2\pi}=\frac{\mathrm{vol}(M)}{2\pi\tau_1}.
      \end{equation}
   Solving the second equation for $p$ and then the first one for $\phi_1$ yields
  	\begin{equation}
  		p=\frac{\mathrm{vol}(M)}{\tau_1\tau_0},\qquad \phi_1=\frac{1}{p}\left(\frac{\tau_1}{\tau_0}-s\right).
  	\end{equation}
  	Similarly, taking the limit as $r\to \rho$ of the second equation of \eqref{eq:a_0_a_1_forms} and substituting \eqref{eq:values_ai} we obtain
  	\begin{equation}
  		\frac{p\tau_0}{2\pi}-\frac{\mathrm{vol}(M)}{2\pi\tau_0}(q+p\phi_0)=0.
  	\end{equation}
   Substituting $\mathrm{vol}(M)=p\tau_0\tau_1$ in the latter equation and solving for $\phi_0$ yields
	\begin{equation}
		\phi_0=\frac{1}{p}\left(\frac{\tau_0}{\tau_1}-q\right).
	\end{equation}
	\end{proof}

\begin{remark}\label{rmk:polar_inverse}
	For $0<r<\rho$ we can write \eqref{eq:Phi_statement} in cylindrical coordinates $(r,\theta,z)$ of \cref{rmk:cyl_coords}
	\begin{equation}
		\exp^\perp_0(\rho-r,-q\theta+mz,p\theta+sz)=\exp^\perp_1(r,\theta,z).
	\end{equation}
\end{remark}
\begin{remark}
	The formula for the volume found in \cref{lem:Phi_and_phis} coincides with the one, proved in larger generality, in \cite{2orbits23}.
\end{remark}
\subsection{Proof of \cref{thm:gather_results}} We combine the three lemmas and prove \cref{thm:gather_results}.

	$(i)$ By \cref{lem:smooth_vol} there are exactly two periodic Reeb orbits, $\ell_0$ and $\ell_1$. Let $\tau_0,\tau_1$ be their minimal periods and $\phi_0,\phi_1$ be their rotation number. By \cref{lem:Phi_and_phis} we have 
	\begin{equation}
		\phi_0=\frac{1}{p}\left(\frac{\tau_0}{\tau_1}-q\right),\qquad \phi_1=\frac{1}{p}\left(\frac{\tau_1}{\tau_0}-s\right).
	\end{equation}
	By \cref{lem:rot_determines_flow} the rotation numbers must be irrational, therefore $\tau_0/\tau_1$ is irrational.\newline\newline\noindent
$(ii)$ By \cref{lem:constant_t_cut} and \cref{lem:smooth_vol} $\ell_0,\ell_1$ are cut loci of each others.  Let $\exp^\perp_i:\nu(\ell_i)\to M$, $i=0,1$, be the corresponding normal exponential maps. By \cref{lem:constant_t_cut} 
\begin{equation}
	\rho=t_{\cut}(\ell_i),\qquad i=0,1.
\end{equation}
Let us denote $U=D^2\times S^1$ and $U_i=B_\rho(\ell_i)$ for $i=0,1$. We define the maps 
\begin{equation}\label{eq:def_psi_i}
	\psi_i:U\to U_i,\qquad \psi_i(r\cos\theta,r\sin\theta,z)=\exp^\perp_i(\rho r\cos\theta,\rho r\sin\theta,z).
\end{equation}
By \cref{thm:cut_locus} the maps $\psi_i$ are diffeomorphisms, and since $\cut(\ell_0)=\ell_1$ and vice versa
\begin{equation}\label{eq:decomposition}
	M=U_0\cup\ell_1=U_1\cup \ell_0.
\end{equation}
In particular $M=U_0\cup U_1$. Note that \eqref{eq:decomposition} implies 
\begin{equation}\label{eq:intersection0}
	\begin{aligned}
		U_0\cap U_1=(M\setminus\ell_1)\cap(M\setminus\ell_0)&=(M\setminus \ell_1)\setminus\ell_0=U_0\setminus \ell_0\\
		&=(M\setminus \ell_0)\setminus\ell_1=U_1\setminus \ell_1.
	\end{aligned}
\end{equation}
Let $\ell=\{0\}\times S^1\subset U$, then, since $\ell_i=\psi_i(\ell)$, equality \eqref{eq:intersection0} can also be written as 
\begin{equation}\label{eq:intersection}
	\psi_0(U\setminus\ell)=\psi_1(U\setminus\ell)=U_0\cap U_1.
\end{equation}
We compute the transition map
\begin{equation}
	\psi:=\psi_0^{-1}\circ\psi_1 :U\setminus\ell\to U\setminus \ell.
\end{equation}
Using the definition of the $\psi_i$'s and \cref{rmk:polar_inverse} yields the following expression of $\psi$ 
\begin{equation}\label{eq:trasnistion_psi}
	\psi(r,\theta,z)=(1-r,-q\theta+mz,p\theta+sz).
\end{equation}
It follows from \eqref{eq:decomposition} that $\{(U_0,\psi_0), (U_1,\psi_1)\}$ is a toroidal atlas for $M$. Moreover, by \eqref{eq:intersection} and \eqref{eq:trasnistion_psi}, such atlas satisfies the hypothesis of \cref{prop:charts}. Thus $M$ is diffeomorphic to $L(p,q)$. 
\newline\newline\noindent
$(iii)$ By \cref{lem:geodesics_and_form} there exists a smooth function $\tilde a:[0,+\infty)\to\mathbb R$ such that 
\begin{equation}
	(\exp^\perp_0)^*\alpha=\left(\frac{\tau_0}{2\pi}-\phi_0\tilde a\left(r\right)\right)dz+\tilde a\left(r\right)d\theta.
\end{equation}
By definition \eqref{eq:def_psi_i} of $\psi_0$ we then have 
\begin{equation}
	\psi_0^*\alpha=	\left(\frac{\tau_0}{2\pi}-\phi_0\tilde a\left(\rho r\right)\right)dz+\tilde a\left(\rho r\right)d\theta,\qquad \forall\,r\in [0,1].
\end{equation}
 The proof can be concluded setting $a(r)=\tilde a(\rho r)$.\qed
\section{K-contact forms on lens spaces}\label{sec:K-cont-Lens}
In this section we construct K-contact forms on $L(p,q)$ and derive necessary and sufficient conditions for them to be contactomorphic. We fix two positive co-prime integers $p,q\in\mathbb Z$, $p>0$. Motivated by point $(iii)$ of \cref{thm:gather_results}, we give the following definition.
\begin{defi}\label{def:triple}
	Let $a:[0,1]\to \mathbb R$ be a smooth function, $\tau_0\in\mathbb R^{+}$ and $\phi_0\in\mathbb R$. A smooth differential form $\alpha$ on $L(p,q)$ is called associated with the triple $(a,\tau_0,\phi_0)$ if 
	\begin{equation}\label{eq:in_def_triple}
		\psi^*_0\alpha=\left(\frac{\tau_0}{2\pi}-\phi_0a(r)\right)dz+a(r)d\theta,
	\end{equation}
	where $(U_0,\psi_0)$ is the toroidal chart of \cref{prop:charts}.
\end{defi}
Combining the latter definition with point $(iii)$ of \cref{thm:gather_results} yields the following corollary. 
\begin{corollary}\label{cor:triple}
	Let $\alpha$ be an irregular K-contact form on $L(p,q)$ having a Reeb orbit of minimal period $\tau_0$ and rotation number $\phi_0$, then $\alpha$ is associated to the triple $(a,\tau_0,\phi_0)$ for some $a\in C^\infty([0,1])$.
\end{corollary}
\subsection{Existence of forms associated to given triples}
In the following three lemmas we derive necessary and sufficient conditions for a smooth differential form $\alpha$ associated to a triple $(a,\tau_0,\phi_0)$ to exists on $L(p,q)$. As in \cref{sec:lens}, we denote $$U=D^2\times S^1\setminus \partial D^2\times S^1,\qquad \ell=\{0\}\times S^1\subset U.$$ 
\begin{lemma}\label{lem:smooth_alpha}
	Let $f,g:[0,1]\to \mathbb R$ be two smooth functions. Let $\alpha$ be a smooth differential form on $U\setminus\ell$ which in cylindrical coordinates reads 
	\begin{equation}
		\alpha=f(r)dz+g(r)d\theta.
	\end{equation}
	Then $\alpha$ extends to a smooth differential form on the whole $U$ if and only if
	\begin{equation}
		g(0)=g^{(2k+1)}(0)=f^{(2k+1)}(0)=0,\qquad \forall k\,\in\mathbb N.
	\end{equation}
\end{lemma}
\begin{proof}
We switch to coordinates $(x,y,z)=(r\cos\theta,r\sin\theta,z)$, which are defined near $r=0$, i.e., near $\ell$. The form rewrites as 
\begin{equation}
	\alpha=f(r)dz+\frac{g(r)}{r^2}(xdy-ydx).
\end{equation}
Thus a necessary and sufficient condition for $\alpha$ to smoothly extend to $r=0$ is that the functions 
\begin{equation}
	F(x,y):=f(\sqrt{x^2+y^2}),\qquad G_1(x,y):=\frac{x g(\sqrt{x^2+y^2})}{x^2+y^2},\qquad G_2(x,y):=\frac{y g(\sqrt{x^2+y^2})}{x^2+y^2},
\end{equation}
are smooth near $x=y=0$. It is a standard fact that a radial function $F(x,y)=f(\sqrt{x^2+y^2})$ is smooth at $(0,0)$ if and only if $f$ is smooth at $0$ and all odd derivatives vanish: $f^{(2k+1)}(0)=0$ for all $k$.
Moreover $G_1,G_2$ are smooth at $(0,0)$ if and only if $g(r)=r^2h(r)$ with $h$ smooth and even at $0$, i.e. $g(0)=0$ and $g^{(2k+1)}(0)=0$ for all $k$.
Hence $\alpha$ extends smoothly across $\ell$ if and only if the stated conditions on $f$ and $g$ hold.
\end{proof}
\begin{lemma}\label{lem:suff_cond}
	Let $a:[0,1]\to \mathbb R$ be a smooth function, $\tau_0>0$ and $\phi_0>-\frac{q}{p}$. A smooth form $\alpha\in\Omega^1(L(p,q))$ associated with the triple $(a,\tau_0,\phi_0)$ exists if and only if the following hold
	\begin{equation}\label{eq:smooth_cond}
		a(0)=a(1)-\frac{p\tau_0}{2\pi(p\phi_0+q)}=0,\qquad a^{(2k+1)}(0)=a^{(2k+1)}(1)=0,\qquad \forall \,\,k\in\mathbb N,
	\end{equation}
	where $a^{(j)}(r)$ denotes the $j$-th derivative of $a$ evaluated at $r\in[0,1]$. 
\end{lemma}
\begin{proof}
	Assume first that a smooth form $\alpha\in \Omega^1(L(p,q))$ associated to the triple $(a,\tau_0,\phi_0)$ exists. According to \cref{prop:charts}, $(U_0,\psi_0)$ is part of toroidal atlas composed by two charts $\{(U_1,\psi_1),(U_0,\psi_0)\}$.
	Let us define the forms 
	\begin{equation}
		\alpha_i=\psi_i^*\alpha,\qquad i=0,1.
	\end{equation}
	By \cref{def:triple} the form $\alpha_0$ is given by
	\begin{equation}\label{eq:alpha0}
		\alpha_0=\left(\frac{\tau_0}{2\pi}-\phi_0a(r)\right)dz+a(r)d\theta.
	\end{equation}
	Since $\alpha_0$ is assumed to be smooth, by \cref{lem:smooth_alpha} we have 
	\begin{equation}
		a(0)=a^{(2k+1)}(0)=0.
	\end{equation}
	By \cref{prop:charts}, $\psi_0(U\setminus\ell)=\psi_1(U\setminus\ell)$. Hence, the restriction of $\alpha_1$ to $U\setminus\ell$ satisfies 
	\begin{equation}
		\alpha_1|_{U\setminus\ell}=(\psi_0^{-1}\circ\psi_1)^*\alpha_0|_{U\setminus\ell}.
	\end{equation}
	Combining the expression of $\psi_0^{-1}\circ \psi_1$ given in \eqref{eq:psi_def} with the one of $\alpha_0$ in \eqref{eq:alpha0}, we obtain
	\begin{equation}\label{eq:alpha1}
		\begin{aligned}
			{\alpha_1}|_{U\setminus\ell}=\left(s\frac{\tau_0}{2\pi}+(m-s\phi_0)a(1-r)\right)dz+\left(p\frac{\tau_0}{2\pi}-(q+p\phi_0)a(1-r)\right)d\theta.
		\end{aligned}
	\end{equation}
	Since $\alpha$ is assumed to be smooth, the form on the right hand side of the latter equality extends smoothly to the whole $U$. Hence, an application of \cref{lem:smooth_alpha} yields
	\begin{equation}
		a(1)=\frac{p\tau_0}{2\pi(p\phi_0+q)},\qquad a^{(2k+1)}(1)=0,\qquad \forall \,\,k\in\mathbb N.
	\end{equation}
	
	Conversely, assume that \eqref{eq:smooth_cond} holds. Then we define the differential forms $\alpha_0$ and $\alpha_1$ as in \eqref{eq:alpha0} and \eqref{eq:alpha1}. Since \eqref{eq:smooth_cond} holds by assumption, \cref{lem:smooth_alpha} ensures that $\alpha_0,\alpha_1$ are smooth on $U$. We define $\alpha$ on $L(p,q)$ as 
	\begin{equation}\label{eq:def_alpha_converse}
		\alpha|_{U_i}=\psi_i^{-1*}\alpha_i,\qquad i=0,1.
	\end{equation}
	Since $L(p,q)=U_0\cup U_1$, \eqref{eq:def_alpha_converse} defines $\alpha$ on the whole and $L(p,q)$. By definition of $\alpha_0,\alpha_1$ 
	\begin{equation}
		(\psi_0^{-1}\circ\psi_1)^*\alpha_0=\alpha_1,
	\end{equation}
	hence the definition \eqref{eq:def_alpha_converse} of $\alpha$ is well posed on the overlap $U_0\cap U_1$.
\end{proof}
\begin{lemma}\label{lem:a_0a_1}
	Let $\tau_0>0$, $\phi_0>-\frac{p}{q}$ and $a:[0,1]\to \mathbb R$ satisfying \eqref{eq:smooth_cond}. Let $\alpha\in\Omega^1(L(p,q))$ be associated to the triple $(a,\tau_0,\phi_0)$. Then there exist $a_0,a_1:[0,1]\to\mathbb R$ smooth functions such that 
	\begin{equation}
		\psi_i^*\alpha=\left(\frac{2\pi}{\tau_i}-\phi_i a_i(r^2)\right)dz+a_i(r^2)dz,\qquad i=0,1,
	\end{equation}
	where the quantities $\tau_1,\phi_1$ are defined as 
	\begin{equation}\label{eq:tau_1,phi_1}
		\tau_1=\frac{\tau_0}{q+p\phi_0},\qquad \phi_1=\frac{m-s\phi_0}{q+p\phi_0},
	\end{equation}
	and $m,s\in\mathbb Z$ such that $mp+sq=1$. Furthermore the functions $a_0,a_1:[0,1]\to\mathbb R$ satisfy
	\begin{equation}\label{eq:boundary_interpol}
		a_0(r^2)=a(r),\qquad a_1(r^2)=p\frac{\tau_0}{2\pi}-\frac{\tau_0}{\tau_1}a_0((1-r)^2),\qquad \forall \, r\in[0,1].
	\end{equation}
\end{lemma}
\begin{proof}
	By \cref{def:triple}, we know that 
	\begin{equation}
		\psi_0^*\alpha=\left(\frac{\tau_0}{2\pi}-\phi_0 a(r)\right)dz+a(r)dz.
	\end{equation}
	Repeating the computation which in the proof of \cref{lem:suff_cond} led to \eqref{eq:alpha1}, we deduce 
	\begin{equation}
		\begin{aligned}
			\psi_1^*\alpha=\left(s\frac{\tau_0}{2\pi}+(m-s\phi_0)a(1-r)\right)dz+\left(p\frac{\tau_0}{2\pi}-(q+p\phi_0)a(1-r)\right)d\theta.
		\end{aligned}
	\end{equation}
	By hypothesis the function $a$ satisfies \eqref{eq:smooth_cond}. In particular all of its odd derivatives vanish at $r=0$ and $r=1$. This implies that there exist two smooth functions $a_0,a_1:[0,1]\to\mathbb R$ such that 
	\begin{equation}
		a_0(r^2)=a(r),\qquad a_1(r^2)=p\frac{\tau_0}{2\pi}-(q+p\phi_0)a(1-r),\qquad \forall\,r\in[0,1].
	\end{equation}
	Inverting the second of the above equations for $a(1-r)$ and substituting the result into the expression of $\psi_1^*\alpha$ we obtain 
	\begin{equation}
		\psi_1^*\alpha=\left(\frac{\tau_0}{2\pi}\left(\frac{sq+mp}{q+p\phi_0}\right)-\left(\frac{m-s\phi_0}{q+p\phi_0}\right)a_1(r^2)\right)dz+a_1(r^2)d\theta.
	\end{equation}
	Combining the fact that $sq+mp=1$ with the definitions of $\tau_1$ and $\phi_1$ given in \eqref{eq:tau_1,phi_1}, we can rewrite the form $\psi_1^*\alpha$ as
	\begin{equation}\label{eq:smooth1}
		\psi_1^*\alpha=\left(\frac{\tau_1}{2\pi}-\phi_1a_1(r^2)\right)dz+a_1(r^2)d\theta.
	\end{equation}
\end{proof}
Contact forms associated to triples $(a,\tau_0,\phi_0)$ are determined by the couple $\tau_0,\phi_0$ alone.
\begin{theorem}\label{lem:triple->cont}
	Let $\tau_0>0$, $\phi_0>-\frac{p}{q}$ and let $a,b:[0,1]\to \mathbb R$ be two smooth functions. Assume that there exists two positive contact forms $\alpha$ and $\beta\in \Omega^1(L(p,q))$ associated to the triples $(a,\tau_0,\phi_0)$ and $(b,\tau_0,\phi_0)$. Then there exists a diffeomorphism $\Psi:L(p,q)\to L(p,q)$ such that 
	\begin{equation}\label{eq:contacto_cond}
		\Psi^*\beta=\alpha.
	\end{equation}
\end{theorem}
\begin{proof}
	By \cref{lem:a_0a_1}, for $i=0,1$ there exist $a_i,b_i:[0,1]\to \mathbb R$ smooth functions such that 
	\begin{equation}
		\psi_i^*\alpha=\left(\frac{2\pi}{\tau_i}-\phi_i a_i(r^2)\right)dz+a_i(r^2)d\theta,\qquad \psi_i^*\beta=\left(\frac{2\pi}{\tau_i}-\phi_i b_i(r^2)\right)dz+b_i(r^2)d\theta,
	\end{equation}
	where the quantities $\tau_1,\phi_1$ are defined in \eqref{eq:tau_1,phi_1}. A standard computation shows that 
	\begin{equation}
		\psi_i^*(\alpha\wedge d\alpha)=\frac{4\pi}{\tau_i} a_i'(r^2)dx \wedge dy\wedge dz,\qquad \psi_i^*(\beta\wedge d\beta)=\frac{4\pi}{\tau_i} b_i'(r^2)dx \wedge dy\wedge dz,
	\end{equation}
	where $(x,y,z)=(r\cos\theta,r\sin\theta,z)$. By hypothesis $\alpha$ and $\beta$ are positive contact forms, therefore 
	\begin{equation}\label{eq:increasing}
		a_i'(r)>0,\qquad b_i'(r)>0,\qquad \forall\,r\in[0,1],\qquad i=0,1.
	\end{equation}
	Furthermore, by \cref{lem:a_0a_1}, both the $a_i$'s and the $b_i$'s satisfy equation \eqref{eq:boundary_interpol}, thus
	\begin{equation}\label{eq:same range}
		a_0(1)=b_0(1)=p\frac{\tau_1}{2\pi},\qquad a_1(1)=b_1(1)=p\frac{\tau_0}{2\pi}\qquad a_i(0)=b_i(0)=0,\qquad i=0,1.
	\end{equation}
	In particular, for each $i=0,1$, $a_i:[0,1]\to \mathbb R$ and $b_i:[0,1]\to\mathbb R$ are immersions with the same range, therefore 
	\begin{equation}
		b_i^{-1}\circ a_i:[0,1]\to [0,1],
	\end{equation}
	is a well defined diffeomorphism. For $i=0,1$ we define the maps 
	\begin{equation}
		\Psi_i:U\to U,\qquad \Psi_i(x,y,z)=\left(\sqrt{b_i^{-1}(a_i(r^2))}\frac{x}{r},\sqrt{b_i^{-1}(a_i(r^2))}\frac{y}{r},z\right).
	\end{equation}
	We claim that $\Psi_i$ is a smooth diffeomorphism. We first prove that $\Psi_i$ is smooth. Let $c_i=\frac{a_i'(0)}{b_i'(0)}$, which is well defined thanks to \eqref{eq:increasing}. By \eqref{eq:same range}, $a_i(0)=b_i(0)=0$, therefore there exists a smooth function $f_i:[0,1]\to\mathbb R$ such that 
	\begin{equation}
		b_i^{-1}(a_i(r^2))=c_i\left(r^2+r^4f_i(r^2)\right).
	\end{equation}
	Therefore we can write 
	\begin{equation}
		\Psi_i(x,y,z)=\left(\sqrt{c_i(1+r^2f_i(r^2))}x,\sqrt{c_i(1+r^2f_i(r^2))}y,z\right),
	\end{equation}
	which proves the smoothness of $\Psi_i$. To see that $\Psi_i$ is actually a diffeomorphism note that its inverse can be explicitly written as
	\begin{equation}
		\Psi_i^{-1}(x,y,z)=\left(\sqrt{a_i^{-1}(b_i(r^2))}\frac{x}{r},\sqrt{a_i^{-1}(b_i(r^2))}\frac{y}{r},z\right),
	\end{equation}
	and its smoothness is understood by the same argument proving the smoothness of $\Psi_i$. We then define $\Psi:L(p,q)\to L(p,q)$ as 
	\begin{equation}
		\Psi|_{U_i}=\psi_i\circ \Psi_i\circ \psi_i^{-1},\qquad i=0,1.
	\end{equation}
	To see that $\Psi$ is well defined, we need to show that
	\begin{equation}\label{eq:cocycle_cond}
		(\psi_0^{-1}\circ\psi_1)^{-1}\circ\Psi_0\circ(\psi_0^{-1}\circ\psi_1)=\Psi_1.
	\end{equation}
	The transition map $\psi:=\psi_0^{-1}\circ \psi_1:U\setminus\ell\to U\setminus \ell$ is computed in \eqref{eq:psi_def}, and in cylindrical coordinates $(r,\theta,z)$ reads
	\begin{equation}
		\psi(r,\theta,z)=(1-r,-q\theta+mz,p\theta+sz),\qquad \psi^{-1}(r,\theta,z)=(1-r,s\theta-mz,-p\theta-qz).
	\end{equation}
	The maps $\Psi_0,\Psi_1$ can also be written in cylindrical coordinates $(r,\theta,z)$ as 
	\begin{equation}\label{eq:polar_Phi}
		\Psi_0(r,\theta,z)=\left(\sqrt{b_0^{-1}(a_0(r^2))},\theta,z\right),\qquad \Psi_1(r,\theta,z)=\left(\sqrt{b_1^{-1}(a_1(r^2))},\theta,z\right).
	\end{equation}
	Therefore we may compute 
	\begin{equation}\label{eq:conjugated_Phi}
		\psi^{-1}\circ\Psi_0\circ\psi(r,\theta,z)=\left(1-\sqrt{b_0^{-1}(a_0((1-r)^2))},\theta,z\right).
	\end{equation}
	By \cref{lem:a_0a_1}, both the $a_i$'s and the $b_i$'s satisfy equation \eqref{eq:boundary_interpol}, in particular 
	\begin{equation}
		a_0((1-r)^2)=-\frac{p\tau_1}{2\pi}-\frac{\tau_1}{\tau_0}a_1(r^2),\qquad b_0((1-r)^2)=-\frac{p\tau_1}{2\pi}-\frac{\tau_1}{\tau_0}b_1(r^2),\qquad \forall r\in [0,1].
	\end{equation}
	From the second equation we find
	\begin{equation}
		r=1-\sqrt{b_0^{-1}\left(-\frac{p\tau_1}{2\pi}-\frac{\tau_1}{\tau_0}b_1(r^2)\right)}.
	\end{equation}
	Evaluating both sides of the latter identity at $\sqrt{b_1^{-1}(a_1(r^2))}$ we find
	\begin{equation}
		\sqrt{b_1^{-1}(a_1(r^2))}=1-\sqrt{b_0^{-1}\left(-\frac{p\tau_1}{2\pi}-\frac{\tau_1}{\tau_0}a_1(r^2)\right)}=1-\sqrt{b_0^{-1}\left(a_0((1-r)^2)\right)}.
	\end{equation}
	Substituting the above identity into \eqref{eq:conjugated_Phi}, in light of the polar coordinate expression of $\Psi_1$ found in \eqref{eq:polar_Phi}, we deduce the validity of \eqref{eq:cocycle_cond}. Therefore $\Psi:L(p,q)\to L(p,q)$ is a well defined smooth diffeomorphism. In  remains to show that \eqref{eq:contacto_cond} holds. Since the chart $U_0$ is dense in $L(p,q)$, it is sufficient to show that 
	\begin{equation}
		\psi_0^*\Psi^*\beta=\psi_0^*\alpha.
	\end{equation}
	Hence, we compute 
	\begin{equation}
		\begin{aligned}
			\psi_0^*\Psi^*\beta&=\psi_0^*\Psi^*\psi_0^{-1*}\psi_0^*\beta=\Psi_0^*\psi_0^*\beta
			=\left(\frac{2\pi}{\tau_0}-\phi_0b_0\circ\Psi_0\right)dz+b_0\circ\Psi_0d\theta\\
			&=\left(\frac{2\pi}{\tau_0}-\phi_0a_0(r^2)\right)dz+a_0(r^2)d\theta=\psi_0^*\alpha,
		\end{aligned}
	\end{equation}
where the fourth equality is obtained from \eqref{eq:polar_Phi}.
\end{proof}
\subsection{K-contact forms with prescribed periods} We now exploit the results of the previous subsection to construct K-contact forms on $L(p,q)$.
\begin{theorem}\label{thm:K-form_on_lens}
	For any $\tau_0,\tau_1>0$ there exists a positive K-contact form on $L(p,q)$ admitting two periodic Reeb orbits $\ell_0,\ell_1$ of minimal periods $\tau_0$, $\tau_1$ and rotation numbers
	\begin{equation}\label{eq:phi01}
		\phi_0=\frac{1}{p}\left(\frac{\tau_0}{\tau_1}-q\right),\qquad \phi_1=\frac{1}{p}\left(\frac{\tau_1}{\tau_0}-s\right),
	\end{equation}
	where $s\in\mathbb Z$, $s\equiv q^{-1}\,\mathrm{mod}\,p$. In particular, if $\tau_0/\tau_1\in\mathbb Q$, then $\alpha$ is quasi-regular, otherwise $\alpha$ is irregular and $\ell_0,\ell_1$ are the only periodic Reeb orbits.
\end{theorem}
\begin{proof}
	Since $\tau_0,\tau_1>0$, then $\phi_0>-\frac{q}{p}$. Let $a:[0,1]\to \mathbb R$ be a smooth function satisfying \eqref{eq:smooth_cond} and the additional condition 
	\begin{equation}\label{eq:contact_cond}
		a'(r)>0\qquad \forall r\in (0,1).
	\end{equation}
	Then according to \cref{lem:suff_cond} there exists a smooth 1-form $\alpha\in\Omega^1(L(p,q))$ associated to the triple $(a,\tau_0,\phi_0)$ in the sense of \cref{def:triple}. Let $m\in\mathbb Z$ such that $mp+sq=1$ and note that $\phi_0$ and $\phi_1$, defined in \eqref{eq:phi01}, satisfy the following equations
	\begin{equation}\label{eq:tau1_phi1_again}
		\tau_1=\frac{\tau_0}{q+p\phi_0},\qquad \phi_1=\frac{m-s\phi_0}{q+p\phi_0}.
	\end{equation}
	Therefore according to \cref{lem:a_0a_1} there are two smooth functions $a_0,a_1:[0,1]\to\mathbb R$ such that 
	\begin{equation}\label{eq:psi*ai}
		\psi_i^*\alpha=\left(\frac{\tau_i}{2\pi}-\phi_i a_i(r^2)\right)dz+a_i(r^2)d\theta,\qquad i=0,1.
	\end{equation}
	A standard computation then shows that 
	\begin{equation}
		\psi_i^*(\alpha\wedge d\alpha)=\frac{4\pi}{\tau_i} a_i'(r^2)dx \wedge dy\wedge dz,\qquad i=0,1.
	\end{equation}
	Combining \eqref{eq:boundary_interpol} and \eqref{eq:contact_cond} we see that $a_i'(r)>0$ for all $r\in [0,1)$, therefore $\alpha$ is a contact form. Let $R$ denote its Reeb vector field of $\alpha$. Computing the Reeb vector field of the form \eqref{eq:psi*ai} yields
	\begin{equation}\label{eq:Reeb_explicit}
		\psi_{i*}^{-1}R=\frac{2\pi}{\tau_i}\left(\partial_z+\phi_i\partial_\theta\right),\qquad i=0,1.
	\end{equation}
	Since $\partial_\theta|_{r=0}=(x\partial_y-y\partial_x)|_{r=0}=0$, in both cases we have that 
	\begin{equation}
		\psi_{i*}^{-1}R|_{r=0}=\frac{2\pi}{\tau_i}\partial_z.
	\end{equation}
	Since $z$ is a $2\pi$-periodic coordinate, we deduce that the curves
	\begin{equation}
		\ell_i:=\psi_i(\ell),\qquad i=0,1,
	\end{equation}
	where we recall that $\ell=\{0\}\times S^1\subset U$, are periodic orbits of $R$ of minimal period $\tau_i$ and rotation numbers $\phi_i$. We see that $\phi_i$ is rational if and only if $\tau_1/\tau_0$ is rational. Therefore from the explict expression of the Reeb vector field \eqref{eq:Reeb_explicit} we deduce that if $\tau_1/\tau_0$ is rational then all Reeb orbits are periodic, while if it is irrational, then $\ell_0$ and $\ell_1$ are the only periodic orbits.
	
	It remains to show that $\alpha$ is K-contact. Let $g$ be a Riemannian metric which is invariant under the $T^2$-action of \cref{prop:T2_action}. Let us denote 
	\begin{equation}
		R_0=\psi_{0*}^{-1}R.
	\end{equation}
	Then, in the notation of \cref{prop:T2_action}, by \eqref{eq:Reeb_explicit} we have 
	\begin{equation}
		\psi_0\circ e^{tR_0}(re^{i\theta},z)
		=\psi_0\left(re^{i\left(\theta+\frac{2\pi \phi_0}{\tau_0}t\right)},z+\frac{2\pi}{\tau_0}t\right)
		=\left(\frac{2\pi\phi_0}{\tau_0}t,\frac{2\pi}{\tau_0}t\right)\cdot \psi_0(re^{i\theta},z),\qquad \forall\,t\in\mathbb R.
	\end{equation}
	Since $e^{tR_0}=\psi_0^{-1}\circ e^{t R}\circ \psi_0$, we deduce that
	\begin{equation}
		e^{tR}\circ\psi_0(re^{i\theta},z)=\left(\frac{2\pi}{\tau_0}t,\frac{2\pi}{\tau_0}\phi_0t\right)\cdot \psi_0(re^{i\theta},z).
	\end{equation}
	Therefore, the flow of the Reeb field is a subgroup of $T^2$. In particular, it holds $\mathcal L_Rg=0$.
\end{proof}
We have the following immediate consequence of \cref{thm:K-form_on_lens}.
\begin{corollary}
	For any $\tau_0,\tau_1\in \mathbb (0,+\infty)$ such that $\frac{\tau_0}{\tau_1}\in\mathbb R\setminus \mathbb Q$ there exists a positive irregular K-contact form on $L(p,q)$ having exactly two periodic orbits $\ell_{\tau_0},\ell_{\tau_1}$ of minimal periods $\tau_0,\tau_1$ respectively and rotation numbers 
	\begin{equation}
		\phi(\ell_{\tau_0})=\frac{1}{p}\left(\frac{\tau_0}{\tau_1}-q\right)\,\,\mathrm{mod}\,\,1,\qquad \phi(\ell_{\tau_1})=\frac{1}{p}\left(\frac{\tau_1}{\tau_0}-s\right)\,\,\mathrm{mod}\,\,1.
	\end{equation}
\end{corollary}
\subsection{Classification of irregular K-contact forms}
In this section we classify irregular K-contact forms. In particular, we prove \cref{thm:classification}. We start with a preliminary lemma. Let $p,q\in\mathbb Z$ be co-prime integers with $p>0$. In the proof we use the notation $\alpha\simeq\alpha'$ to indicate that two contact forms $\alpha$ and $\alpha'$ are strictly contactomorphic.
\begin{lemma}\label{lem:first_order_inv}
	Let $\alpha$ and $\alpha'$ be positive irregular $K$-contact forms on $L(p,q)$.
	Then $\alpha\simeq\alpha'$ if and only if there exist
	periodic Reeb orbits $\ell$ of $\alpha$ and $\ell'$ of $\alpha'$ having the
	same minimal period and the same rotation number.
\end{lemma}
\begin{proof}
	Let $\alpha,\alpha'$ be two irregular K-contact forms on $L(p,q)$ admitting a periodic Reeb orbit of minimal period $\tau_0$ and rotation number $\phi_0$. By \cref{cor:triple} there exist two smooth functions $a,a':[0,1]\to\mathbb R$ such that, in the sense of \cref{def:triple}, $\alpha$ and $\alpha'$ are associated to triples $(a,\tau_0,\phi_0)$ and $(a',\tau_0,\phi_0)$, respectively. By \cref{lem:triple->cont}, $\alpha$ and $\alpha'$ are contactomorphic.
\end{proof}
We can now prove \cref{thm:classification}. 
\begin{proof}[Proof of \cref{thm:classification}]
	By \cref{thm:K-form_on_lens}, for any $\tau_0,\tau_1>0$ with
	$\tau_0/\tau_1 \in \mathbb{R}\setminus\mathbb{Q}$, there exists an irregular
	$K$-contact form on $L(p,q)$ with exactly two periodic Reeb orbits, denoted by
	$\ell_{\tau_0}$ and $\ell_{\tau_1}$, whose minimal periods are $\tau_0$ and $\tau_1$,
	respectively, and whose rotation numbers are
	\begin{equation}
		\phi(\ell_{\tau_0})=\frac{1}{p}\left(\frac{\tau_0}{\tau_1}-q\right),\qquad
		\phi(\ell_{\tau_1})=\frac{1}{p}\left(\frac{\tau_1}{\tau_0}-s\right),
	\end{equation}
 where $sq\equiv 1\,\mathrm{mod}\,p$. We denote such K-contact form with $\alpha(\tau_0,\tau_1)$. Exchanging the role of $\tau_0$ and $\tau_1$, we deduce that there exists another K-contact form, which is $\alpha(\tau_1,\tau_0)$ on $L(p,q)$ having exactly two periodic orbits $\bar{\ell}_{\tau_0},\bar{\ell}_{\tau_1}$ of minimal periods $\tau_0,\tau_1$ and rotation numbers 
 \begin{equation}
 	\phi(\bar{\ell}_{\tau_0})=\frac{1}{p}\left(\frac{\tau_0}{\tau_1}-s\right),\qquad \phi(\bar{\ell}_{\tau_1})=\frac{1}{p}\left(\frac{\tau_1}{\tau_0}-q\right).
 \end{equation}
We claim that 
\begin{equation}\label{eq:equiv_q_s}
	\alpha(\tau_0,\tau_1)\simeq\alpha(\tau_1,\tau_0)\iff q^2\equiv 1\,\mathrm{mod}\,p.
\end{equation} 
Indeed, by \cref{lem:first_order_inv}, $\alpha(\tau_0,\tau_1)\simeq\alpha(\tau_1,\tau_0)$ if and only if the periodic Reeb orbits of the two contact forms have identical minimal periods and rotation numbers. Since $\tau_0/\tau_1$ is irrational, it follows that $\tau_0 \neq \tau_1$. Consequently, for $\alpha(\tau_0,\tau_1)\simeq\alpha(\tau_1,\tau_0)$ to hold it is necessary, and by \cref{lem:first_order_inv} it is also sufficient, that 
\begin{equation}\label{eq:equiv_phi}
	\phi({\ell}_{\tau_0}) \equiv \phi(\bar{\ell}_{\tau_0}) \mod{1}.
\end{equation}
The latter condition is equivalent to $q\equiv s\,\mathrm{mod}\,p$, or to $q^2\equiv 1\,\mathrm{mod}\,p$.

To conclude the proof, we now show that any irregular K-contact form on $L(p,q)$ is contactomorphic to $\alpha(\tau_0,\tau_1)$, for some choice of $\tau_0,\tau_1>0$ with irrational ratio. Let $\alpha$ be a positive irregular K-contact form on $L(p,q)$. By \cref{thm:gather_results} $\alpha$ admits exactly two periodic Reeb orbits $\ell_0$, $\ell_1$ of minimal periods $\bar\tau_0,\bar\tau_1>0$ with irrational ratio. Furthermore, by \cref{lem:Phi_and_phis}, these orbits have rotation numbers 
\begin{equation}
	\phi_0=\frac{1}{p}\left(\frac{\bar\tau_0}{\bar\tau_1}-q\right),\qquad \phi_1=\frac{1}{p}\left(\frac{\bar\tau_1}{\bar\tau_0}-s\right).
\end{equation}
It follows by \cref{lem:first_order_inv} that $\alpha\simeq\alpha(\bar\tau_0,\bar\tau_1).$
\end{proof}
\section{The trace of the heat kernel}\label{sec:trace_heat}
In this section we prove \cref{prop:spectral_inv} and \cref{thm:spectral_inv}.
\subsection{Sub-Laplacian and heat kernel}
Let $(M,\alpha,g)$ be a compact contact Riemannian manifold (cf. \cref{sec:preliminaries}). The sub-Riemannian gradient of $f\in C^\infty(M)$, denoted with $\nabla^{sR} f$, is defined as the unique smooth section of $\xi=\ker\alpha$ satisfying 
\begin{equation}
	df(Y)=g(\nabla^{sR} f,Y),\qquad \forall\, Y\in\Gamma(\xi).
\end{equation}
The sub-Laplacian of $f$ is then defined as the divergence of its sub-Riemannian gradient 
\begin{equation}
	\Delta^{sR} f=\mathrm{div}(\nabla ^{sR} f),
\end{equation}
where the divergence is taken with respect to the contact volume form $\alpha\wedge d\alpha$. Such operator is hypoelliptic \cite{Hormander67}, therefore for any $\phi\in C^\infty(M)$, any solution $u$ to the equation $\Delta^{sR} u=\phi$ is smooth. For a given $\varphi\in C^{\infty}(M)$, consider the following initial value problem 
\begin{equation}\label{eq:heat_equation}
	\begin{cases}
		\partial_t \psi(t,x)=\Delta^{sR} \psi(t,x),\qquad &(t,x)\in \mathbb R^+\times M,\\
		\lim_{t\to 0}\psi(t,x)=\varphi(x),\qquad &x\in M.
	\end{cases}
\end{equation}
Problem \eqref{eq:heat_equation} has a unique smooth solution which, for any $\varphi\in C^\infty(M)$, can be written as
\begin{equation}
	\psi(t,x)=\int_M P(t,x,y)\varphi(y)\alpha\wedge d\alpha(y),\qquad (t,x)\in\mathbb R^+\times M,
\end{equation}
where the function $P\in C^\infty(\mathbb R^+\times M\times M)$ is the so called \emph{heat kernel}. See \cite[Chap.\,21]{Agrachev} and references therein for a proof of the existence of a smooth heat kernel. We are interested in the asymptotic expansion as $t\to 0$ of the trace of the heat kernel, which is
\begin{equation}
	\mathbb R^+\ni t\mapsto \int_MP(t,x,x)\alpha\wedge d\alpha(x).
\end{equation}
In \cite{Barilari2013} the first two coefficients of the expansion of $P(t,x,x)$ are computed. The main result of \cite{BenArous89} ensures that, for compact $M$, the expansion found in \cite{Barilari2013} can be integrated in $x$. The following theorem follows.
\begin{theorem}[\cite{Barilari2013,BenArous89}]\label{thm:Davide}
	Let $(M,\alpha,g)$ be a contact Riemannian manifold and let $$P:\mathbb R^+\times M\times M\to\mathbb R,$$ be the heat kernel associated to the sub-Laplacian. Then, as $t\to 0$, the following holds 
	\begin{equation}\label{eq:heat-trace}
		\int_M P(t,x,x)\alpha\wedge d\alpha(x)=\frac{1}{16t^2}\left(\int_M \alpha\wedge d\alpha+t\int_M \kappa \alpha\wedge d\alpha +O(t^2)\right),
	\end{equation}
	where $\kappa$ is the sectional curvature of $\ker\alpha$ computed with the Tanno connection \cite{Blair,Tanno89}.
\end{theorem}
In \cite{Barilari2013} $\kappa$ is presented as a local sub-Riemannian invariant, without mention of the Tanno connection. Lemma 13 of \cite{BarBesLer2020} ensures that the latter sub-Riemannian local invariant coincides with the sectional curvature of $\ker\alpha$ computed with the Tanno connection.
\subsection{Quasi regular case}
In this subsection we compute the coefficients of the expansion \eqref{eq:heat-trace} for compact K-contact Riemannian manifolds $(M,\alpha,g)$ with $\alpha$ quasi-regular. We recall that in this case the Reeb flow of $\alpha$ induces a Seifert fibration of $M$ (cf. \cref{app:siefert}). We refer the reader to \cite[Sec.\,4]{CristofaroMazz2020} for a concise treatment of Seifert fibrations induced by Reeb flows. 
\begin{proof}[Proof of \cref{prop:spectral_inv}]
	Since $\alpha$ is quasi-regular, its Reeb flow induces a Seifert fibration of $M$ (cf. \cref{prop:Seifert_Reeb}). In particular all Reeb orbits are periodic and all but finitely many of them have the same minimal period $\tau>0$. Then by \cite[Cor.\,6.3]{Geig2022} or \cite[Lemma 3.2]{AbbLangeMazz2022}, one has
	\begin{equation}\label{eq:quasi-reg-vol}
		\int_M\alpha\wedge d\alpha=-\tau e(M),
	\end{equation}
	where $e(M)$ is the Euler number of the Seifert fibration (cf. \cref{app:siefert}). 
	
	Let $M^\times$ denotes the complement in $M$ of the union of the Reeb orbits of minimal period different from $\tau$. Let $\mathcal O$ be the quotient of $M$ and let $\mathcal O^\times$ be the quotient of $M^\times$, by the Reeb flow. Note that, since the Reeb flow induces a Siefert firbation of $M$, $\mathcal O$ is a smooth orbifold whereas $\mathcal O^\times$ is a smooth surface, which is the complement of the singular points of $\mathcal O$ (cf. \cref{app:siefert}). Furthermore $M^\times$ is a trivial $S^1$-bundle over $\mathcal O^\times$:
	\begin{equation}
		M^{\times}\simeq \mathcal O^\times \times S^1.
	\end{equation}
    Therefore, there exist coordiantes $(x,t)\in \mathcal O \times S^1$ such that 
    \begin{equation}
    	R=\frac{2\pi}{\tau}\partial_t.
    \end{equation}
	Since $\mathcal L_R\alpha=0$ and $\alpha(R)=1$, there exists a smooth 1-form $\beta$ defined on $\mathcal O^\times$ such that 
	\begin{equation}
		\alpha=\frac{\tau}{2\pi}dt+\beta.
	\end{equation}
	In particular, it holds 
	\begin{equation}
		\alpha\wedge d\alpha=\frac{\tau}{2\pi}dt\wedge d\beta.
	\end{equation}
	Let $\kappa$ be the sectional curvature of $\xi$ computed with respect to the Tanno connection. Since $R$ is a Killing vector field for $g$ we have $R(\kappa)=0$, therefore in coordinates $(x,t)$ we have $\kappa=\kappa(x)$. Hence we can compute 
	\begin{equation}\label{eq:comp-curv}
		\begin{aligned}
			\int_M\kappa \alpha\wedge d\alpha&=\int_{M^\times}\kappa \alpha\wedge d\alpha\int_{\mathcal O^\times\times S^1}\kappa(x)\frac{\tau}{2\pi}dt\wedge d\beta(x)=\\
		&=\int_0^{2\pi}\frac{\tau}{2\pi}dt\int_{\mathcal O^\times}\kappa(x)d\beta(x)=\tau\int_{\mathcal O^\times}\kappa(x)d\beta(x).
	\end{aligned}
	\end{equation}
 Note that, since $g$ is invariant under the Reeb flow, $\mathcal O^\times$ inherits a natural Riemannian structure. Moreover $d\beta$ is the Riemannian area form of $\mathcal O^\times$. By \cite[Lemma 13]{BarBesLer2020} $\kappa$ is the local sub-Riemannian invariant of \cite[Remark 2.8]{ABBR2024}. Thus, by the latter remark, $\kappa$ is the Gaussian curvature of $\mathcal O^\times$. By the orbifold Gauss-Bonnet theorem \cite[Chap.\,13]{Thurston22},\cite{Satake57} we have 
 \begin{equation}
 	\int_{\mathcal O^\times}\kappa(x) d\beta(x)=2\pi\chi_{\mathrm{orb}}(\mathcal O).
 \end{equation}
Substituting the latter equation in \eqref{eq:comp-curv} we obtain 
\begin{equation}\label{eq:quasi-reg-curv}
	\int_M\kappa \alpha\wedge d\alpha=2\pi\tau\chi_{\mathrm{orb}}(\mathcal O).
\end{equation}
Combining \eqref{eq:quasi-reg-vol}, \eqref{eq:quasi-reg-curv} and \cref{thm:Davide}, we conclude the proof.
	\end{proof}
\subsection{Irregular case} In this subsection we compute the coefficients of the expansion \eqref{eq:heat-trace} for compact K-contact Riemannian manifolds $(M,\alpha,g)$ with $\alpha$ irregular. 
\begin{lemma}\label{lem:chi_seifert_Lpq}
	Let $p,q$ be coprime integers, $p>0$. Let $\alpha$ be a quasi-regular contact form on $L(p,q)$. Then there exist two natural numbers $a_0,a_1$ and $\tau>0$ such that the set of minimal periods of the periodic Reeb orbits is 
	\begin{equation}\label{eq:spetrum_Lpq}
		\left\{\tau,\frac{\tau}{a_0},\frac{\tau}{a_1}\right\}.
	\end{equation}
	 Furthermore, there exist at most two Reeb orbits of minimal period less than $\tau$, and are called the exceptional fibers. Moreover, if there are exactly two exceptional fibers, the Euler charateristic of the base orbifold of the Seifert fibration induced by the Reeb flow is 
	\begin{equation}
		\chi_{\mathrm{orb}}(\mathcal O)=\frac{1}{a_0}+\frac{1}{a_1}.
	\end{equation}
\end{lemma}
\begin{proof}
The fact that the set of minimal periods of periodic Reeb orbits has the form \eqref{eq:spetrum_Lpq} follows from \cref{prop:Seifert_Reeb} and \cite[Sec.\,4.4]{GeigesLange2018}. Furthermore by \cite{GeigesLange2018}, if there are exactly two exceptional fiber then the base orbifold of the Seifert fibration is the football $S(a_0,a_1)$, that is, the 2-dimensional orbifold whose underlying topological surface is the 2-sphere and has exactly two cone singularities of order $a_0,a_1\in\mathbb N$. The orbifold Euler characteristic (cf. \cref{app:siefert}) is then
\begin{equation}
	\chi_{\mathrm{orb}}(S(a_0,a_1))=\chi(S^2)-\sum_{i=0}^1\left(1-\frac{1}{a_i}\right)=2-\sum_{i=0}^1\left(1-\frac{1}{a_i}\right)=\frac{1}{a_0}+\frac{1}{a_1}.
\end{equation}
\end{proof}
\begin{lemma}\label{lem:deformation}
	Let $(M,\alpha,g)$ be a K-contact Riemannian manifold with $\alpha$ irregular and let $\ell_0,\ell_1$ be the only two periodic Reeb orbits, of minimal periods $\tau_0,\tau_1$ (cf. \cref{thm:main1}). Then there exists a sequence of smooth contact forms $\{\alpha_n\}$ and Riemannian metrics $\{g_n\}$ on $M$ such that
	\begin{itemize}
		\item[(i)] $(M,\alpha_n,g_n)$ is a K-contact Riemannian manifold for all $n\in\mathbb N$,
		\item[(ii)] $\alpha_n$ is quasi-Regular and $\ell_0,\ell_1$ are the only exceptional fibers of the Seifert fibration induced by its Reeb flow (see \cref{app:siefert}), for all $n\in\mathbb N$,
		\item[(iii)] The following limits hold true in the $C^\infty$-topology
		\begin{equation}
			\lim_{n\to\infty}\alpha_n=\alpha,\qquad \lim_{n\to\infty} g_n=g.
		\end{equation}
	Denoting with $\tau_n(\ell_i)$ the minimal period of $\ell_i$ with respect to the Reeb field of $\alpha_n$, it holds 
	\begin{equation}
		\lim_{n\to\infty}\tau_n(\ell_i)=\tau_i,\qquad i=0,1.
	\end{equation}
	\end{itemize}
\end{lemma}
\begin{proof}
	Let $\delta:M\to\mathbb R$ be the distance function from $\ell_0$ computed with respect to the metric $g$. Let $\mu:M\to\mathbb R$ be the volume function 
	\begin{equation}
		\mu(x)=\int_{\delta<\delta(x)}\alpha\wedge d\alpha,\qquad x\in M.
	\end{equation}
	According to \cref{lem:smooth_vol}, $\mu$ is a smooth K-contact momentum map. For $\ve\in \mathbb R^+$, we define the following 1-form 
	\begin{equation}
		\alpha_\ve=\frac{1}{1+\ve \mu}\alpha.
	\end{equation}
 	Since $\mu$ is a smooth K-contact momentum map, according to \cref{rmk:deformation}, $\alpha_\ve$ is a K-contact form and there exists a Reeb invariant compatible metric $g_\ve$ such that, in the $C^\infty$-topology
 	\begin{equation}\label{eq:lim_alpha_g}
 		\lim_{\ve\to 0}\alpha_\ve=\alpha,\qquad \lim_{\ve\to 0}g_\ve=g.
 	\end{equation}
 	We now compute the differential of $\alpha_\ve$
 	\begin{equation}
 		d\alpha_\ve=-\frac{\ve d\mu}{(1+\ve\mu)^2}\wedge \alpha+\frac{1}{1+\ve\mu}d\alpha.
 	\end{equation}
 	According to \cref{rmk:mu_01}, $\ell_0$ and $\ell_1$ are both composed of critical points of $\mu$, therefore 
 	\begin{equation}
 		d\alpha_\ve|_{\ell_i}=\frac{1}{1+\ve\mu(\ell_i)}d\alpha|_{\ell_i},\qquad \alpha_{\ve}|_{\ell_i}=\frac{1}{1+\ve\mu({\ell_i})}\alpha|_{\ell_i},\qquad i=0,1.
 	\end{equation}
 	Since the Reeb field $R_\ve$ of the form $\alpha_\ve$ is defined by 
 	\begin{equation}
 		d\alpha_\ve(R_\ve,\cdot)=0,\qquad \alpha_\ve(R_\ve)=1,
 	\end{equation}
 	we deduce that 
 	\begin{equation}
 		{R_\ve}|_{\ell_i}=(1+\ve\mu({\ell_i}))R|_{\ell_i},\qquad i=1,2,
 	\end{equation}
 	where $R$ is the Reeb field of $\alpha$. Hence $\ell_i$ are both orbits of the perturbed Reeb field $R_\ve$, which, along them, is a constant rescaling of $R$. In particular 
 	\begin{equation}\label{eq:perturbed_periods}
 		\tau_\ve(\ell_0)=\frac{\tau_0}{1+\ve\mu({\ell_0})}=\tau_0,\qquad \tau_\ve(\ell_1)=\frac{\tau_1}{1+\ve\mu({\ell_1})}=\frac{\tau_1}{1+\ve p\tau_0\tau_1},
 	\end{equation}
 	where in the latter equality we have used the fact that 
 	\begin{equation}
 		\mu(\ell_0)=0,\qquad \mu(\ell_1)=\mathrm{vol}(M)=p\tau_0\tau_1,
 	\end{equation}
 	obtained from \cref{rmk:mu_01} and \cref{lem:Phi_and_phis}. Thus, there exists a sequence ${\ve_n}$ of positive real numbers such that 
 	\begin{equation}
 		\lim_{n\to\infty} \ve_n=0,\qquad \frac{\tau_{\ve_n}(\ell_0)}{\tau_{\ve_n}(\ell_1)}\in\mathbb Q,\qquad \forall\,n\in\mathbb N.
 	\end{equation}
  We now prove that the sequences $\{\alpha_{\ve_n}\}, \{g_{\ve_n}\}$ of forms and metrics enjoy the desired properties. 
  \newline\newline\noindent
  $(i)$ We have already proved that $\alpha_{\ve_n}$ is K-contact and $g_{\ve_n}$ is Reeb invariant under the Reeb flow of $\alpha_{\ve_n}$ and compatible with it.
  \newline\newline\noindent
  $(ii)$ Since $\tau_{\ve_n}(\ell_0)/\tau_{\ve_n}(\ell_1)\in\mathbb Q$, by \cref{thm:main1} $\alpha_{\ve_n}$ cannot be irregular, and it is therefore quasi-regular. Let $\ell\subset M$ be a (periodic) Reeb orbit for $\alpha_{\ve_n}$ different from $\ell_0,\ell_1$. By \cref{rmk:deformation}, $\mu$ is a K-contact momentum map for $(M,\alpha_{\ve_n},g_{\ve_n})$, and, by \cref{rmk:mu_01}, $\ell$ is contained in the regular set of $\mu$ (i.e. $d\mu|_\ell\neq 0$) . Thus by \cref{lem:R(f)=0}, $\ell$ has zero rotation number. Consequently, by \cref{lem:rot_determines_flow}, $\ell$ has an open neighborhood where all Reeb orbits have its same period. Therefore, according to \cref{lem:chi_seifert_Lpq}, $\ell$ is not an exceptional fiber (recall that $M$ is diffeomorphic to a lens space by \cref{thm:main1}). Since $\tau_0/\tau_1$ is irrational, by \eqref{eq:perturbed_periods}, we can assume that nor $\tau_{\ve_n}(\ell_0)/\tau_{\ve_n}(\ell_1)$, nor its reciprocal $\tau_{\ve_n}(\ell_1)/\tau_{\ve_n}(\ell_0)$, are natural numbers for $n$ big enough. If one among $\ell_0$ and $\ell_1$ were not exceptional, than at least one of the aforementioned ratios would be a natural number (cf. \cref{lem:chi_seifert_Lpq}). We deduce that both $\ell_0$ and $\ell_1$ are exceptional.
   \newline\newline\noindent
  $(iii)$ This follows by \eqref{eq:lim_alpha_g} and \eqref{eq:perturbed_periods}.
\end{proof}
We can now prove \cref{thm:spectral_inv}.
\begin{proof}[Proof of \cref{thm:spectral_inv}]
		We compute both coefficients appearing in the expansion \eqref{eq:heat-trace}. Concerning the total volume, by \cref{lem:Phi_and_phis} it holds 
		\begin{equation}\label{eq:total_vol_irreg}
			\int_M\alpha\wedge d\alpha=p\tau_0\tau_1.
		\end{equation}
	Let $\alpha_n,g_n$ be the sequence of contact forms and Riemannian metrics of \cref{lem:deformation}.  In particular $\alpha_n$ is quasi-regular and $g_n$ is compatible with $\alpha_n$ and invariant under its Reeb flow. Let $\kappa_n$ be the sectional curvature of $\ker\alpha_n$ computed with respect to the Tanno connection, then, combining \cref{thm:Davide}, \cref{prop:spectral_inv} and \cref{lem:chi_seifert_Lpq}, we find that 
	\begin{equation}\label{eq:total_k_n}
			\int_M\kappa_n\alpha_n\wedge d\alpha_n=2\pi(\tau_n(\ell_0)+\tau_n(\ell_1)).
	\end{equation}
 Since $\alpha_n\to \alpha$ and $g_n\to g$ in the $C^\infty$-topology, then $\kappa_n\alpha_n\wedge d\alpha_n\to \kappa \alpha\wedge d\alpha$. Therefore, taking the limit of both sides of \eqref{eq:total_k_n}, thanks to point $(iii)$ of \cref{lem:deformation} we obtain 
 \begin{equation}\label{eq:total_k_irreg}
 	\int_M\kappa\alpha\wedge d\alpha=2\pi(\tau_0+\tau_1).
 \end{equation}
 Substituting \eqref{eq:total_vol_irreg} and \eqref{eq:total_k_irreg} in \eqref{eq:heat-trace}, we conclude the proof.
\end{proof}
\newpage
\appendix
\section{Seifert fibrations}\label{app:siefert}

We briefly recall some basic facts concerning Seifert fibrations, with particular emphasis on those induced by Reeb flows. 
The notion of Seifert fibration is much more general than the one introduced here (see \cite{Orlik}).

A Seifert fibration of a closed oriented three-manifold is a decomposition of $M$ into circles determined by a locally free circle action.

\begin{defi}
	Let $M$ be a closed oriented $3$-manifold with an effective locally free $S^1$-action, 
	that is, an effective $S^1$-action such that all isotropy groups (stabilizers of points) are finite. 
	The quotient map
	\begin{equation}
		\pi \colon M \to \mathcal O := M/S^1
	\end{equation}
	is called a Seifert fibration of $M$.
\end{defi}

The \emph{fibers} of a Seifert fibration are the fibers of $\pi$, which are precisely the orbits of the $S^1$-action. 
One can show that the set of points with nontrivial isotropy group is a finite union of fibers;
these are called the \emph{exceptional fibers}. 
Let $\ell$ be one of the latter. 
Since the finite subgroups of $S^1$ are necessarily cyclic, there exists $a \in \mathbb N$ such that the isotropy group of each point in $\ell$ is isomorphic to $\mathbb Z_a$. 
We call $a$ the \emph{order} of the exceptional fiber $\ell$. 
Note that in the absence of exceptional fibers a Seifert fibration reduces to a principal $S^1$-bundle.

The quotient $\mathcal O$ is a 2-dimensional \emph{orbifold}. 
It is a topological surface, smooth outside of a finite collection of points. 
In our case, these points are precisely the projections of the exceptional fibers and are cone singularities whose orders coincide with the orders of the corresponding exceptional fibers. 
The \emph{orbifold Euler characteristic} of $\mathcal O$ is defined as
\begin{equation}
	\chi_{\mathrm{orb}}(\mathcal O)
	=
	\chi(\mathcal O)
	-
	\sum_{i=1}^n \left(1 - \frac{1}{a_i}\right),
\end{equation}
where $\chi(\mathcal O)$ denotes the Euler characteristic of the underlying topological surface and $a_1,\dots,a_n$ are the orders of the cone singularities. 
Two Seifert fibrations are called \emph{isomorphic} if there exists a diffeomorphism sending the fibers of the first fibration to those of the second. 
The orbifold Euler characteristic $\chi_{\mathrm{orb}}(\mathcal O)$ is a Seifert invariant, i.e. invariant under isomorphism.

Another invariant of a Seifert fibration is the \emph{Seifert Euler number}, which generalizes the Euler number of a principal $S^1$-bundle. 
In particular, if there are no exceptional fibers, the Seifert fibration is a principal $S^1$-bundle and its Seifert Euler number coincides with the classical one. 
For a detailed definition of this invariant and for a complete description of Seifert invariants, we refer the reader to \cite{Orlik,Hatcher2002}.

The relevance of Seifert fibrations for our analysis stems from Wadsley's theorem \cite{Wadsley75}, which implies that the Reeb flow of a quasi-regular contact form induces a Seifert fibration. 
As a consequence, we have the following proposition (see, for instance, \cite[Sec.~4]{CristofaroMazz2020}).

\begin{prop}\label{prop:Seifert_Reeb}
	Let $(M,\alpha)$ be a closed three-manifold with a quasi-regular contact form. 
	Then the Reeb flow induces a Seifert fibration of $M$. 
	In particular, all but finitely many Reeb orbits have the same minimal period $\tau > 0$. 
	The remaining ones are the exceptional fibers of the Seifert fibration and have minimal periods $\tau/a_1,\dots,\tau/a_n$, where $a_i$ is the order of the $i$-th exceptional fiber.
\end{prop}

\bibliographystyle{alphaabbr}
	\bibliography{bib_rudimetal.bib}

@article {CristofaroMazz2020,
	AUTHOR = {Cristofaro-Gardiner, Daniel and Mazzucchelli, Marco},
	TITLE = {The action spectrum characterizes closed contact 3-manifolds
		all of whose {R}eeb orbits are closed},
	JOURNAL = {Comment. Math. Helv.},
	FJOURNAL = {Commentarii Mathematici Helvetici. A Journal of the Swiss
		Mathematical Society},
	VOLUME = {95},
	YEAR = {2020},
	NUMBER = {3},
	PAGES = {461--481},
	ISSN = {0010-2571,1420-8946},
	MRCLASS = {53D10},
	MRNUMBER = {4152621},
	MRREVIEWER = {Ruishi\ Kuwabara},
	DOI = {10.4171/CMH/493},
	URL = {https://doi.org/10.4171/CMH/493},
}

@book {Orlik,
	AUTHOR = {Orlik, Peter},
	TITLE = {Seifert manifolds},
	SERIES = {Lecture Notes in Mathematics, Vol. 291},
	PUBLISHER = {Springer-Verlag, Berlin-New York},
	YEAR = {1972},
	PAGES = {viii+155},
	MRCLASS = {57E15},
	MRNUMBER = {426001},
	MRREVIEWER = {Richard\ Randell},
}

@misc{Hatcher2002,
	author       = {Allen Hatcher},
	title        = {Notes on Basic 3--Manifold Topology},
	note         = {Unpublished lecture notes},
	year         = {2002},
	howpublished = {\url{https://pi.math.cornell.edu/~hatcher/3M/3M.pdf}}
}

@article {Boothby,
	AUTHOR = {Boothby, W. M. and Wang, H. C.},
	TITLE = {On contact manifolds},
	JOURNAL = {Annals of Mathematics},
	FJOURNAL = {},
	VOLUME = {68},
	YEAR = {1958},
	NUMBER = {3},
	PAGES = {721--734},
}

@book{BoyerGalicki2008,
	author    = {Charles P. Boyer and Krzysztof Galicki},
	title     = {Sasakian Geometry},
	series    = {Oxford Mathematical Monographs},
	publisher = {Oxford University Press},
	address   = {Oxford},
	year      = {2008},
	isbn      = {978-0-19-856495-9}
}

@book {DoCarmo,
	AUTHOR = {Manfredo P. do Carmo},
	TITLE = {Riemannian Geometry},
	SERIES = {Mathematics: Theory and Applications},
	VOLUME = {},
	PUBLISHER = {Birkhäuser Boston, MA},
	YEAR = {1992},
}

@article {Massot,
	AUTHOR = {Etnyre, J.B. and Komendarczyk, R. and Massot, P.},
	TITLE = {Tightness in contact metric 3-manifolds.},
	JOURNAL = {Invent. math.},
	FJOURNAL = {},
	VOLUME = {188},
	YEAR = {2012},
	NUMBER = {2},
	PAGES = {621--657},
	
}

@book {Blair,
	AUTHOR = {Blair, D.E.},
	TITLE = {Riemannian Geometry of Contact and Symplectic Manifolds.},
	SERIES = {Progress in Mathematics},
	VOLUME = {203},
	PUBLISHER = {Birkhäuser Boston},
	YEAR = {2010},
}

@book {Agrachev,
	AUTHOR = {Agrachev, Andrei and Barilari, Davide and Boscain, Ugo},
	TITLE = {A comprehensive introduction to sub-{R}iemannian geometry},
	SERIES = {Cambridge Studies in Advanced Mathematics},
	VOLUME = {181},
	NOTE = {From the Hamiltonian viewpoint,
	With an appendix by Igor Zelenko},
	PUBLISHER = {Cambridge University Press, Cambridge},
	YEAR = {2020},
	
}

@article {ABBR2024,
	AUTHOR = {Agrachev, Andrei A and Baranzini, Stefano and Bellini, Eugenio
	and Rizzi, Luca},
	TITLE = {Quantitative tightness for three-dimensional contact
	manifolds: a sub-{R}iemannian approach},
	JOURNAL = {Nonlinearity},
	FJOURNAL = {Nonlinearity},
	VOLUME = {38},
	YEAR = {2025},
	NUMBER = {11},
	PAGES = {Paper No. 115011, 62},
	ISSN = {0951-7715,1361-6544},
	MRCLASS = {53D10 (53C17 57K33)},
	MRNUMBER = {4989350},
	DOI = {10.1088/1361-6544/ae19be},
	URL = {https}}

@article {Wadsley75,
	AUTHOR = {Wadsley, A. W.},
	TITLE = {Geodesic foliations by circles},
	JOURNAL = {J. Differential Geometry},
	FJOURNAL = {Journal of Differential Geometry},
	VOLUME = {10},
	YEAR = {1975},
	NUMBER = {4},
	PAGES = {541--549},
	ISSN = {0022-040X,1945-743X},
	MRCLASS = {57D30},
	MRNUMBER = {400257},
	MRREVIEWER = {Bruce\ L.\ Reinhart},
	URL = {http://projecteuclid.org/euclid.jdg/1214433160},
}

@article {Sullivan78,
	AUTHOR = {Sullivan, Dennis},
	TITLE = {A foliation of geodesics is characterized by having no
		``tangent homologies''},
	JOURNAL = {J. Pure Appl. Algebra},
	FJOURNAL = {Journal of Pure and Applied Algebra},
	VOLUME = {13},
	YEAR = {1978},
	NUMBER = {1},
	PAGES = {101--104},
	ISSN = {0022-4049,1873-1376},
	MRCLASS = {57R30},
	MRNUMBER = {508734},
	MRREVIEWER = {Robert\ Roussarie},
	DOI = {10.1016/0022-4049(78)90046-4},
	URL = {https://doi.org/10.1016/0022-4049(78)90046-4},
}

@article {CrisHutc2016,
	AUTHOR = {Cristofaro-Gardiner, Daniel and Hutchings, Michael},
	TITLE = {From one {R}eeb orbit to two},
	JOURNAL = {J. Differential Geom.},
	FJOURNAL = {Journal of Differential Geometry},
	VOLUME = {102},
	YEAR = {2016},
	NUMBER = {1},
	PAGES = {25--36},
	ISSN = {0022-040X,1945-743X},
	MRCLASS = {53D10 (53D42)},
	MRNUMBER = {3447085},
	MRREVIEWER = {Cecilia\ Karlsson},
	URL = {http://projecteuclid.org/euclid.jdg/1452002876},
}

@article {GeigesLange2018,
	AUTHOR = {Geiges, Hansj\"{o}rg and Lange, Christian},
	TITLE = {Seifert fibrations of lens spaces},
	JOURNAL = {Abh. Math. Semin. Univ. Hambg.},
	FJOURNAL = {Abhandlungen aus dem Mathematischen Seminar der
	Universit\"{a}t Hamburg},
	VOLUME = {88},
	YEAR = {2018},
	NUMBER = {1},
	PAGES = {1--22},
	ISSN = {0025-5858,1865-8784},
	MRCLASS = {57M50 (55R65 57M10 57M60)},
	MRNUMBER = {3785783},
	MRREVIEWER = {Daniel\ Ruberman},
	DOI = {10.1007/s12188-017-0188-z},
	URL = {https://doi.org/10.1007/s12188-017-0188-z},
}

@article {2orbits23,
	AUTHOR = {Cristofaro-Gardiner, Daniel and Hryniewicz, Umberto and
	Hutchings, Michael and Liu, Hui},
	TITLE = {Contact three-manifolds with exactly two simple {R}eeb orbits},
	JOURNAL = {Geom. Topol.},
	FJOURNAL = {Geometry \& Topology},
	VOLUME = {27},
	YEAR = {2023},
	NUMBER = {9},
	PAGES = {3801--3831},
	ISSN = {1465-3060,1364-0380},
	MRCLASS = {53D42 (37J55 53E50)},
	MRNUMBER = {4674840},
	MRREVIEWER = {Alexander\ Fel\cprime shtyn},
	DOI = {10.2140/gt.2023.27.3801},
	URL = {https://doi.org/10.2140/gt.2023.27.3801},
}

@article {BrokenBook23,
	AUTHOR = {Colin, Vincent and Dehornoy, Pierre and Rechtman, Ana},
	TITLE = {On the existence of supporting broken book decompositions for
	contact forms in dimension 3},
	JOURNAL = {Invent. Math.},
	FJOURNAL = {Inventiones Mathematicae},
	VOLUME = {231},
	YEAR = {2023},
	NUMBER = {3},
	PAGES = {1489--1539},
	ISSN = {0020-9910,1432-1297},
	MRCLASS = {57K33 (37C27 37C35 53E50)},
	MRNUMBER = {4549092},
	MRREVIEWER = {Rostislav\ Matveyev},
	DOI = {10.1007/s00222-022-01160-7},
	URL = {https://doi.org/10.1007/s00222-022-01160-7},
}

@article {TorsionForms2019,
	AUTHOR = {Cristofaro-Gardiner, Dan and Hutchings, Michael and
	Pomerleano, Daniel},
	TITLE = {Torsion contact forms in three dimensions have two or
	infinitely many {R}eeb orbits},
	JOURNAL = {Geom. Topol.},
	FJOURNAL = {Geometry \& Topology},
	VOLUME = {23},
	YEAR = {2019},
	NUMBER = {7},
	PAGES = {3601--3645},
	ISSN = {1465-3060,1364-0380},
	MRCLASS = {53D10 (53D42)},
	MRNUMBER = {4047649},
	MRREVIEWER = {Georgios\ Dimitroglou Rizell},
	DOI = {10.2140/gt.2019.23.3601},
	URL = {https://doi.org/10.2140/gt.2019.23.3601},
}

@article {AlbersGeigesKai2022,
	AUTHOR = {Albers, Peter and Geiges, Hansj\"{o}rg and Zehmisch, Kai},
	TITLE = {Pseudorotations of the 2-disc and {R}eeb flows on the
	3-sphere},
	JOURNAL = {Ergodic Theory Dynam. Systems},
	FJOURNAL = {Ergodic Theory and Dynamical Systems},
	VOLUME = {42},
	YEAR = {2022},
	NUMBER = {2},
	PAGES = {402--436},
	ISSN = {0143-3857,1469-4417},
	MRCLASS = {37E30 (37J55 53D35)},
	MRNUMBER = {4362897},
	MRREVIEWER = {Frans\ Cantrijn},
	DOI = {10.1017/etds.2021.15},
	URL = {https://doi.org/10.1017/etds.2021.15},
}

@incollection {Katok24,
	AUTHOR = {Katok, A. B.},
	TITLE = {Ergodic perturbations of degenerate integrable {H}amiltonian
	systems},
	BOOKTITLE = {The collected works of {A}natole {K}atok. {V}ol. 1},
	PAGES = {613--649},
	NOTE = {Reprint of English translation of [0331425]},
	PUBLISHER = {World Sci. Publishing, Singapore},
	YEAR = {[2024] \copyright 2024},
	ISBN = {[9789811237751]; [9789811237768]; [9789811238062]},
	MRCLASS = {37J40 (28D99)},
	MRNUMBER = {4893880},
}

@incollection {FayadKatok24,
	AUTHOR = {Fayad, Bassam and Katok, Anatole},
	TITLE = {Constructions in elliptic dynamics},
	BOOKTITLE = {The collected works of {A}natole {K}atok. {V}ol. 1},
	PAGES = {825--873},
	NOTE = {Corrected reprint of [2104594]},
	PUBLISHER = {World Sci. Publishing, Singapore},
	YEAR = {[2024] \copyright 2024},
	ISBN = {[9789811237751]; [9789811237768]; [9789811238062]},
	MRCLASS = {37C05 (37C40 37E30)},
	MRNUMBER = {4893886},
}

@article {Geig2022,
	AUTHOR = {Geiges, Hansj\"{o}rg},
	TITLE = {What does a vector field know about volume?},
	JOURNAL = {J. Fixed Point Theory Appl.},
	FJOURNAL = {Journal of Fixed Point Theory and Applications},
	VOLUME = {24},
	YEAR = {2022},
	NUMBER = {2},
	PAGES = {Paper No. 23, 26},
	ISSN = {1661-7738,1661-7746},
	MRCLASS = {57R30 (37J55 53C22 53D35 57R25 58A10)},
	MRNUMBER = {4403698},
	MRREVIEWER = {Shigenori\ Matsumoto},
	DOI = {10.1007/s11784-022-00946-9},
	URL = {https://doi.org/10.1007/s11784-022-00946-9},
}

@article {AbbLangeMazz2022,
	AUTHOR = {Abbondandolo, Alberto and Lange, Christian and Mazzucchelli,
	Marco},
	TITLE = {Higher systolic inequalities for 3-dimensional contact
	manifolds},
	JOURNAL = {J. \'{E}c. polytech. Math.},
	FJOURNAL = {Journal de l'\'{E}cole polytechnique. Math\'{e}matiques},
	VOLUME = {9},
	YEAR = {2022},
	PAGES = {807--851},
	ISSN = {2429-7100,2270-518X},
	DOI = {10.5802/jep.195},
	URL = {https://doi.org/10.5802/jep.195},
}

@article {KegelLange2021,
	AUTHOR = {Kegel, Marc and Lange, Christian},
	TITLE = {A {B}oothby-{W}ang theorem for {B}esse contact manifolds},
	JOURNAL = {Arnold Math. J.},
	FJOURNAL = {Arnold Mathematical Journal},
	VOLUME = {7},
	YEAR = {2021},
	NUMBER = {2},
	PAGES = {225--241},
	ISSN = {2199-6792,2199-6806},
	MRCLASS = {57R18 (53D35 57R17)},
	MRNUMBER = {4260074},
	MRREVIEWER = {Philippe\ Rukimbira},
	DOI = {10.1007/s40598-020-00165-5},
	URL = {https://doi.org/10.1007/s40598-020-00165-5},
}

@article {Hormander67,
	AUTHOR = {Hormander, Lars},
	TITLE = {Hypoelliptic second order differential equations},
	JOURNAL = {Acta Math.},
	FJOURNAL = {Acta Mathematica},
	VOLUME = {119},
	YEAR = {1967},
	PAGES = {147--171},
	ISSN = {0001-5962,1871-2509},
	MRCLASS = {35.48 (47.00)},
	MRNUMBER = {222474},
	MRREVIEWER = {Joel\ Smoller},
	DOI = {10.1007/BF02392081},
	URL = {https://doi.org/10.1007/BF02392081},
}

@article {Tanno89,
	AUTHOR = {Tanno, Shukichi},
	TITLE = {Variational problems on contact {R}iemannian manifolds},
	JOURNAL = {Trans. Amer. Math. Soc.},
	FJOURNAL = {Transactions of the American Mathematical Society},
	VOLUME = {314},
	YEAR = {1989},
	NUMBER = {1},
	PAGES = {349--379},
	ISSN = {0002-9947,1088-6850},
	MRCLASS = {53C15 (32F25 58G30)},
	MRNUMBER = {1000553},
	MRREVIEWER = {Gerhard\ Huisken},
	DOI = {10.2307/2001446},
	URL = {https://doi.org/10.2307/2001446},
}

@article {Barilari2013,
	AUTHOR = {Barilari, D.},
	TITLE = {Trace heat kernel asymptotics in 3{D} contact sub-{R}iemannian
	geometry},
	NOTE = {Translation of Sovrem. Mat. Prilozh. No. 82 (2012)},
	JOURNAL = {J. Math. Sci. (N.Y.)},
	FJOURNAL = {Journal of Mathematical Sciences (New York)},
	VOLUME = {195},
	YEAR = {2013},
	NUMBER = {3},
	PAGES = {391--411},
	ISSN = {1072-3374,1573-8795},
	MRCLASS = {58J35 (35K08 35R01 53C17)},
	MRNUMBER = {3207127},
	MRREVIEWER = {Jing\ Wang},
	DOI = {10.1007/s10958-013-1585-1},
	URL = {https://doi.org/10.1007/s10958-013-1585-1},
}

@article {ClassicalQuantum23,
	AUTHOR = {Colin de Verdi\`ere, Yves},
	TITLE = {Classical and quantum mechanics on 3{D} contact manifolds},
	JOURNAL = {Pure Appl. Math. Q.},
	FJOURNAL = {Pure and Applied Mathematics Quarterly},
	VOLUME = {19},
	YEAR = {2023},
	NUMBER = {4},
	PAGES = {1839--1852},
	ISSN = {1558-8599,1558-8602},
	MRCLASS = {58J50 (53C17 53D10)},
	MRNUMBER = {4671383},
	DOI = {10.4310/pamq.2023.v19.n4.a5},
	URL = {https://doi.org/10.4310/pamq.2023.v19.n4.a5},
}

@article {Schwartz63,
	AUTHOR = {Schwartz, Arthur J.},
	TITLE = {A generalization of a {P}oincar\'{e}-{B}endixson theorem to
	closed two-dimensional manifolds},
	JOURNAL = {Amer. J. Math.},
	FJOURNAL = {American Journal of Mathematics},
	VOLUME = {85},
	YEAR = {1963},
	PAGES = {453--458; errata: 85 (1963), 753},
	ISSN = {0002-9327,1080-6377},
	MRCLASS = {34.65 (57.48)},
	MRNUMBER = {155061},
	MRREVIEWER = {Bruce\ L.\ Reinhart},
}

@article {Satake57,
	AUTHOR = {Satake, Ichir\^{o}},
	TITLE = {The {G}auss-{B}onnet theorem for {$V$}-manifolds},
	JOURNAL = {J. Math. Soc. Japan},
	FJOURNAL = {Journal of the Mathematical Society of Japan},
	VOLUME = {9},
	YEAR = {1957},
	PAGES = {464--492},
	ISSN = {0025-5645,1881-1167},
	MRCLASS = {53.00},
	MRNUMBER = {95520},
	MRREVIEWER = {C.\ B.\ Allendoerfer},
	DOI = {10.2969/jmsj/00940464},
	URL = {https://doi.org/10.2969/jmsj/00940464},
}

@book {Thurston22,
	AUTHOR = {Thurston, William P.},
	TITLE = {The geometry and topology of three-manifolds. {V}ol. {IV}},
	NOTE = {Edited and with a preface by Steven P. Kerckhoff and a chapter
	by J. W. Milnor},
	PUBLISHER = {American Mathematical Society, Providence, RI},
	YEAR = {[2022] \copyright 2022},
	PAGES = {xvii+316},
	ISBN = {978-1-4704-6391-5; [9781470468361]; [9781470451646]},
	MRCLASS = {57K32 (53C15 57R30)},
	MRNUMBER = {4554426},
	MRREVIEWER = {Thilo\ Kuessner},
}

@article {Taubes2007,
	AUTHOR = {Taubes, Clifford Henry},
	TITLE = {The {S}eiberg-{W}itten equations and the {W}einstein
	conjecture},
	JOURNAL = {Geom. Topol.},
	FJOURNAL = {Geometry \& Topology},
	VOLUME = {11},
	YEAR = {2007},
	PAGES = {2117--2202},
	ISSN = {1465-3060,1364-0380},
	MRCLASS = {57R17 (53D35 57R57 58J30)},
	MRNUMBER = {2350473},
	MRREVIEWER = {Hans\ U.\ Boden},
	DOI = {10.2140/gt.2007.11.2117},
	URL = {https://doi.org/10.2140/gt.2007.11.2117},
}

@misc{SpectralAsymp2022,
	title={Spectral asymptotics for sub-Riemannian Laplacians}, 
	author={Colin de Verdi\`ere, Yves and Hillairet, Luc and Tr\'{e}lat,
	Emmanuel},
	year={2022},
	eprint={2212.02920},
	archivePrefix={arXiv},
	primaryClass={math.DG},
	url={https://arxiv.org/abs/2212.02920}, 
}

@article {SpectralAsym1,
	AUTHOR = {Colin de Verdi\`ere, Yves and Hillairet, Luc and Tr\'{e}lat,
	Emmanuel},
	TITLE = {Spectral asymptotics for sub-{R}iemannian {L}aplacians, {I}:
	{Q}uantum ergodicity and quantum limits in the 3-dimensional
	contact case},
	JOURNAL = {Duke Math. J.},
	FJOURNAL = {Duke Mathematical Journal},
	VOLUME = {167},
	YEAR = {2018},
	NUMBER = {1},
	PAGES = {109--174},
	ISSN = {0012-7094,1547-7398},
	MRCLASS = {58J37 (35P20 35R01 53C17 58J51)},
	MRNUMBER = {3743700},
	MRREVIEWER = {Jingzhi\ Tie},
	DOI = {10.1215/00127094-2017-0037},
	URL = {https://doi.org/10.1215/00127094-2017-0037},
}

@article {BenArous89,
	AUTHOR = {Ben Arous, G\'{e}rard},
	TITLE = {D\'{e}veloppement asymptotique du noyau de la chaleur
	hypoelliptique sur la diagonale},
	JOURNAL = {Ann. Inst. Fourier (Grenoble)},
	FJOURNAL = {Universit\'{e} de Grenoble. Annales de l'Institut Fourier},
	VOLUME = {39},
	YEAR = {1989},
	NUMBER = {1},
	PAGES = {73--99},
	ISSN = {0373-0956,1777-5310},
	MRCLASS = {58G32 (22E30 35K05 58G11 60H07)},
	MRNUMBER = {1011978},
	MRREVIEWER = {Mireille\ Chaleyat-Maurel},
	DOI = {10.5802/aif.1158},
	URL = {https://doi.org/10.5802/aif.1158},
}

@article{BenArous88,
	author  = {Ben Arous, G{\'e}rard},
	title   = {D{\'e}veloppement asymptotique du noyau de la chaleur hypoelliptique hors du cut-locus},
	journal = {Annales scientifiques de l'{\'E}cole Normale Sup{\'e}rieure},
	series  = {S{\'e}rie 4},
	volume  = {21},
	number  = {3},
	pages   = {307--331},
	year    = {1988},
	doi     = {10.24033/asens.1560}
}

@article{ShapeDrum,
	ISSN = {00029890, 19300972},
	URL = {http://www.jstor.org/stable/2313748},
	author = {Mark Kac},
	journal = {The American Mathematical Monthly},
	number = {4},
	pages = {1--23},
	publisher = {[Taylor & Francis, Ltd., Mathematical Association of America]},
	title = {Can One Hear the Shape of a Drum?},
	urldate = {2026-01-20},
	volume = {73},
	year = {1966}
}

@book {Rosenberg,
	AUTHOR = {Rosenberg, Steven},
	TITLE = {The {L}aplacian on a {R}iemannian manifold},
	SERIES = {London Mathematical Society Student Texts},
	VOLUME = {31},
	NOTE = {An introduction to analysis on manifolds},
	PUBLISHER = {Cambridge University Press, Cambridge},
	YEAR = {1997},
	PAGES = {x+172},
	ISBN = {0-521-46300-9; 0-521-46831-0},
	MRCLASS = {58Gxx (35J05 35K05)},
	MRNUMBER = {1462892},
	MRREVIEWER = {Friedbert\ Pr\"{u}fer},
	DOI = {10.1017/CBO9780511623783},
	URL = {https://doi.org/10.1017/CBO9780511623783},
}

@article {BarBesLer2020,
	AUTHOR = {Barilari, Davide and Beschastnyi, Ivan and Lerario, Antonio},
	TITLE = {Volume of small balls and sub-{R}iemannian curvature in 3{D}
	contact manifolds},
	JOURNAL = {J. Symplectic Geom.},
	FJOURNAL = {The Journal of Symplectic Geometry},
	VOLUME = {18},
	YEAR = {2020},
	NUMBER = {2},
	PAGES = {355--384},
	ISSN = {1527-5256,1540-2347},
	MRCLASS = {53C17 (53D10)},
	MRNUMBER = {4118144},
	MRREVIEWER = {Luca\ Rizzi},
	DOI = {10.4310/jsg.2020.v18.n2.a1},
	URL = {https://doi.org/10.4310/jsg.2020.v18.n2.a1},
}

@article {BellBosc2023,
	AUTHOR = {Bellini, Eugenio and Boscain, Ugo},
	TITLE = {Surfaces of genus {$g\geq1$} in 3{D} contact sub-{R}iemannian
	manifolds},
	JOURNAL = {ESAIM Control Optim. Calc. Var.},
	FJOURNAL = {ESAIM. Control, Optimisation and Calculus of Variations},
	VOLUME = {29},
	YEAR = {2023},
	PAGES = {Paper No. 79, 8},
	ISSN = {1292-8119,1262-3377},
	MRCLASS = {53D10 (49Q20 53C17)},
	MRNUMBER = {4664851},
	MRREVIEWER = {Haiming\ Liu},
	DOI = {10.1051/cocv/2023072},
	URL = {https://doi.org/10.1051/cocv/2023072},
}

@article {Bell2023,
	AUTHOR = {Bellini, Eugenio},
	TITLE = {The geometry of {R}iemannian curvature radii},
	JOURNAL = {J. Dyn. Control Syst.},
	FJOURNAL = {Journal of Dynamical and Control Systems},
	VOLUME = {29},
	YEAR = {2023},
	NUMBER = {4},
	PAGES = {1829--1853},
	ISSN = {1079-2724,1573-8698},
	MRCLASS = {53C17 (20F45)},
	MRNUMBER = {4673183},
	MRREVIEWER = {Alexey\ V.\ Podobryaev},
	DOI = {10.1007/s10883-023-09664-y},
	URL = {https://doi.org/10.1007/s10883-023-09664-y},
}

@article {BarBoscCann2022,
	AUTHOR = {Barilari, Davide and Boscain, Ugo and Cannarsa, Daniele},
	TITLE = {On the induced geometry on surfaces in 3{D} contact
	sub-{R}iemannian manifolds},
	JOURNAL = {ESAIM Control Optim. Calc. Var.},
	FJOURNAL = {ESAIM. Control, Optimisation and Calculus of Variations},
	VOLUME = {28},
	YEAR = {2022},
	PAGES = {Paper No. 9, 28},
	ISSN = {1292-8119,1262-3377},
	MRCLASS = {53C17 (53A05)},
	MRNUMBER = {4371078},
	MRREVIEWER = {Haiming\ Liu},
	DOI = {10.1051/cocv/2021104},
	URL = {https://doi.org/10.1051/cocv/2021104},
}

@misc{BarBelPina2025,
	author        = {Barilari, Davide and Bellini, Eugenio and Pinamonti, Andrea},
	title         = {Curvature measures and the sub-Riemannian {G}auss--{B}onnet theorem},
	year          = {2025},
	eprint        = {2509.26460},
	archivePrefix = {arXiv},
	primaryClass  = {math.DG},
	doi           = {10.48550/arXiv.2509.26460},
	url           = {https://arxiv.org/abs/2509.26460},
	note          = {arXiv:2509.26460}
}

@article {AgrachevRizziRossi24,
	AUTHOR = {Agrachev, Andrei and Rizzi, Luca and Rossi, Tommaso},
	TITLE = {Relative heat content asymptotics for sub-{R}iemannian
	manifolds},
	JOURNAL = {Anal. PDE},
	FJOURNAL = {Analysis \& PDE},
	VOLUME = {17},
	YEAR = {2024},
	NUMBER = {9},
	PAGES = {2997--3037},
	ISSN = {2157-5045,1948-206X},
	MRCLASS = {53C17 (35K15 35R01 58J35 58J60)},
	MRNUMBER = {4818197},
	MRREVIEWER = {Robert\ Weston\ Neel},
	DOI = {10.2140/apde.2024.17.2997},
	URL = {https://doi.org/10.2140/apde.2024.17.2997},
}

@article {Strichartz86,
	AUTHOR = {Strichartz, Robert S.},
	TITLE = {Sub-{R}iemannian geometry},
	JOURNAL = {J. Differential Geom.},
	FJOURNAL = {Journal of Differential Geometry},
	VOLUME = {24},
	YEAR = {1986},
	NUMBER = {2},
	PAGES = {221--263},
	ISSN = {0022-040X,1945-743X},
	MRCLASS = {53C20 (53A40 53C21 53C22 58G30)},
	MRNUMBER = {862049},
	MRREVIEWER = {Karsten\ Grove},
	URL = {http://projecteuclid.org/euclid.jdg/1214440436},
}

@article {RizziRossi21,
	AUTHOR = {Rizzi, Luca and Rossi, Tommaso},
	TITLE = {Heat content asymptotics for sub-{R}iemannian manifolds},
	JOURNAL = {J. Math. Pures Appl. (9)},
	FJOURNAL = {Journal de Math\'{e}matiques Pures et Appliqu\'{e}es.
	Neuvi\`eme S\'{e}rie},
	VOLUME = {148},
	YEAR = {2021},
	PAGES = {267--307},
	ISSN = {0021-7824,1776-3371},
	MRCLASS = {53C17 (35R01 58J35 58J60)},
	MRNUMBER = {4223354},
	MRREVIEWER = {Jing\ Wang},
	DOI = {10.1016/j.matpur.2020.12.004},
	URL = {https://doi.org/10.1016/j.matpur.2020.12.004},
}

@article{Rossi21,
	author = {Tommaso Rossi},
	title = {The {Relative} {Heat} {Content} for {Submanifolds} in {Sub-Riemannian} {Geometry}},
	journal = {S\'eminaire de th\'eorie spectrale et g\'eom\'etrie},
	pages = {191--212},
	year = {2019-2021},
	publisher = {Institut Fourier},
	address = {Grenoble},
	volume = {36},
	doi = {10.5802/tsg.376},
	language = {en},
	url = {https://proceedings.centre-mersenne.org/articles/10.5802/tsg.376/}
}

@article {TysonWang18,
	AUTHOR = {Tyson, Jeremy and Wang, Jing},
	TITLE = {Heat content and horizontal mean curvature on the {H}eisenberg
	group},
	JOURNAL = {Comm. Partial Differential Equations},
	FJOURNAL = {Communications in Partial Differential Equations},
	VOLUME = {43},
	YEAR = {2018},
	NUMBER = {3},
	PAGES = {467--505},
	ISSN = {0360-5302,1532-4133},
	MRCLASS = {53C17 (35K05 35R03 58J65)},
	MRNUMBER = {3804205},
	MRREVIEWER = {Andrea\ Pinamonti},
	DOI = {10.1080/03605302.2018.1446166},
	URL = {https://doi.org/10.1080/03605302.2018.1446166},
}

@article {BarBoscNeel2012,
	AUTHOR = {Barilari, Davide and Boscain, Ugo and Neel, Robert W.},
	TITLE = {Small-time heat kernel asymptotics at the sub-{R}iemannian cut
	locus},
	JOURNAL = {J. Differential Geom.},
	FJOURNAL = {Journal of Differential Geometry},
	VOLUME = {92},
	YEAR = {2012},
	NUMBER = {3},
	PAGES = {373--416},
	ISSN = {0022-040X,1945-743X},
	MRCLASS = {58J35 (53Cxx)},
	MRNUMBER = {3005058},
	MRREVIEWER = {Weiyong\ He},
	URL = {http://projecteuclid.org/euclid.jdg/1354110195},
}

@book {Sakai1996,
	AUTHOR = {Sakai, Takashi},
	TITLE = {Riemannian geometry},
	SERIES = {Translations of Mathematical Monographs},
	VOLUME = {149},
	NOTE = {Translated from the 1992 Japanese original by the author},
	PUBLISHER = {American Mathematical Society, Providence, RI},
	YEAR = {1996},
	PAGES = {xiv+358},
	ISBN = {0-8218-0284-4},
	MRCLASS = {53-01 (53-02)},
	MRNUMBER = {1390760},
	MRREVIEWER = {Conrad\ Plaut},
	DOI = {10.1090/mmono/149},
	URL = {https://doi.org/10.1090/mmono/149},
}
\author{Eugenio Bellini,}
\address{Dipartimento di Matematica "Tullio Levi-Civita", Universit\`a degli Studi di Padova, via Trieste 63, 35131 Padova (PD), Italy. email:  \textbf{\emph{eugenio.bellini@unipd.it}}}
\end{document}